\documentclass[a4paper,10pt]{amsart}

\usepackage{amssymb,latexsym}
\usepackage{amsthm}
\usepackage{amsmath}
\usepackage{amsfonts}
\usepackage{mathrsfs}
\usepackage{mathdots}
\usepackage{graphicx}
\usepackage[all]{xy}
\usepackage{chngcntr}
\usepackage{fancyhdr}

\usepackage[T1]{fontenc}
\usepackage{lmodern}

\usepackage{tikz-cd}
\usepackage{stmaryrd}

\newtheorem{theorem}{Theorem}[section]
\newtheorem{definition}[theorem]{Definition}
\newtheorem{remark}[theorem]{Remark}
\newtheorem{conjecture}[theorem]{Conjecture}
\newtheorem{proposition}[theorem]{Proposition}
\newtheorem{corollary}[theorem]{Corollary}
\newtheorem{lemma}[theorem]{Lemma}
\newtheorem{example}[theorem]{Example}

\def\Q{\mathbb{Q}}
\def\F{\mathbb{F}}

\def\Z{\mathbb{Z}}

\def\C{\mathbb{C}}
\def\G{\mathbb{G}}

\def\E{\mathcal{E}}

\def\Fc{\mathcal{F}}

\def\I{\mathcal{I}}

\def\M{\mathcal{M}}

\def\Ol{\mathcal{O}}

\def\V{\mathcal{V}}

\def\Bun{\mathrm{Bun}}

\def\Fil{\mathrm{Fil}}
\def\Hom{\mathrm{Hom}}
\def\Lie{\mathrm{Lie}}

\def\Mod{\mathrm{Mod}}

\def\GL{\mathrm{GL}}
\def\Gal{\mathrm{Gal}}

\def\Proj{\mathrm{Proj}}
\def\Rep{\mathrm{Rep}}

\def\Spa{\mathrm{Spa}}

\def\Spec{\mathrm{Spec}}
\def\Vect{\mathrm{Vect}}
\def\deg{\mathrm{deg}}

\def\lan{\langle}
\def\ran{\rangle}

\def\lra{\longrightarrow}
\def\ra{\rightarrow}
\def\ov{\overline}

\def\wt{\widetilde}
\def\st{\stackrel}
\def\tr{\textrm}

\def\phmod{\varphi\mathrm{-mod}}
\def\Fb{\breve{F}}
\def\Fbar{\overline{F}}
\def\a{\alpha}

\def\gc{\gamma^\vee}
\def\Phic{\Phi^\vee}
\def\iso{\xrightarrow{\ \sim\ }}
\def\DD{\mathbb{D}}
\def\ph{\varphi}
\def\drt{\rightarrow}
\def\Fun{F^{\mathrm{un}}}
\def\s{\sigma}
\def\Gr{\mathrm{Gr}}
\def\BdR{\mathrm{B_{dR}}}
\def\BpdR{\mathrm{B_{dR}^+}}
\def\TT{\boldsymbol{T}}
\def\AA{\boldsymbol{A}}

\def\e{\epsilon}
\def\et{\text{\'et}}

\usepackage[pdftex,colorlinks=true,linkcolor=red,
            urlcolor=orange,
            citecolor=blue]{hyperref}

\begin{document}

\title{On the structure of some p-adic period domains}
\author{Miaofen Chen, Laurent Fargues and Xu Shen}
\date{}
\address{Department of Mathematics\\
	Shanghai Key Laboratory of PMMP\\
	East China Normal University\\
	No. 500, Dong Chuan Road\\
	Shanghai 200241, China}\email{mfchen@math.ecnu.edu.cn}
\address{CNRS\\
	Institut de Math\'ematiques de Jussieu\\
	4 place Jussieu\\
	 Paris 75252, France}\email{laurent.fargues@imj-prg.fr}

\address{Morningside Center of Mathematics, Academy of Mathematics and Systems Science, Chinese Academy of Sciences\\
	No. 55, Zhongguancun East Road\\
	Beijing 100190, China}\email{shen@math.ac.cn}
	
\thanks{Miaofen Chen was partially supported by NSFC grant No.11671136 and STCSM grant  No.13dz2260400. 
Laurent Fargues was partially supported by ANR grant  ANR-14-CE25-0002 "PerCoLaTor"  and ERC Advanced grant  742608 "GeoLocLang". 
Xu Shen was partially supported by the Chinese Academy of Sciences grants 50Y64198900 and 29Y64200900. }

\renewcommand\thefootnote{}

\renewcommand{\thefootnote}{\arabic{footnote}}

\begin{abstract}
	We prove the Fargues-Rapoport conjecture for $p$-adic period domains: for a reductive group $G$ over a $p$-adic field and a minuscule cocharacter $\mu$ of $G$, the weakly admissible locus coincides with the admissible one if and only if the Kottwitz set $B(G,\mu)$ is fully Hodge-Newton decomposable.
\end{abstract}

\maketitle
\setcounter{tocdepth}{1}
\tableofcontents

\section*{Introduction}
Let $F$ be a finite degree extension of $\Q_p$,
$G$  a connected reductive group over $F$, and $\{\mu\}$  the geometric conjugacy class of a minuscule cocharacter  $\mu$ of $G$. Attached to the datum $(G, \{\mu\})$, we have the flag variety $\Fc(G,\mu)$ defined over the reflex field  $E=E(G,\{\mu\})$, the field of definition of $\{\mu\}$, a finite degree extension of $F$. We will consider the associated adic space 
$\breve{\Fc} (G,\mu)$ 
over $\breve{E}$,  the completion of the maximal unramified extension of $E$. After fixing an element $b\in G(\breve{F})$,  Rapoport and Zink constructed in \cite{RZ} (see also \cite{Ra, DOR})  an open subspace \[\breve{\Fc}(G, \mu, b)^{wa}\] inside $\breve{\Fc}(G,\mu)$, which is called  a $p$-adic period domain, as a vast generalization of  Drinfeld upper half spaces (\cite{Dr}). The name ``$p$-adic period domain'' comes as follows. For any finite extension $K|\breve{E}$, the points in $\Fc(G, \mu, b)^{wa}(K)$ correspond to weakly admissible filtered isocrystals equipped with a $G$-structure, which are then admissible (as $K|\breve{E}$ is finite) by a fundamental result in $p$-adic Hodge theory (\cite{CoFo}). We thus get  crystalline representations with additional structures attached to points in $\breve{\Fc}(G, \mu, b)^{wa}(K)$. Here in order to get non empty $\breve{\Fc}(G, \mu, b)^{wa}$, we have to assume that $b$ satisfies a certain condition with respect to $\{\mu\}$, cf. prop. \ref{P:wa non empty} for some background on this.

If $K|\breve{E}$ is an arbitrary complete extension, then it is not clear whether we still get Galois representations attached to points in $\breve{\Fc}(G, \mu, b)^{wa}(K)$. This lead Rapoport and Zink to conjecture that there exists an open subspace \[\breve{\Fc}(G, \mu, b)^{a}\] inside $\breve{\Fc}(G, \mu, b)^{wa}$, such that there exists a  $p$-adic local system with additional structures over $\breve{\Fc}(G, \mu, b)^{a}$ which interpolates the crystalline representations attached to all classical points, cf. \cite{DOR} conj. 11.4.4 and \cite{Har} conj. 2.3. Contrary to $\breve{\Fc}(G, \mu, b)^{wa}$, it is difficult to give a direct construction (and explicit description) for the desired $\breve{\Fc}(G, \mu, b)^{a}$.

 For certain triples $(G, \mu, b)$ (those so called of PEL type), Hartl (\cite{Har, Har1}) and Faltings (\cite{Fal}) constructed the space $\breve{\Fc}(G, \mu, b)^{a}$  by using the Robba ring $\wt{B}^\dagger_{rig}$  and the crystalline period ring $B_{cris}$ respectively. If $K|\breve{E}$ is finite, we have $$\breve{\Fc}(G, \mu, b)^{a}(K)=\breve{\Fc}(G, \mu, b)^{wa}(K)$$ as explained above. But it turns out that in general the inclusion \[\breve{\Fc}(G, \mu, b)^{a}\subset \breve{\Fc}(G, \mu, b)^{wa}\] is strict.
\\

By the recent progress in $p$-adic Hodge theory, thanks to the works of Fargues-Fontaine \cite{FF} and Fargues \cite{F3}, we can now construct  $\breve{\Fc}(G, \mu, b)^{a}$ as in the Rapoport-Zink conjecture for any triple $(G, \{\mu\}, b)$ (compatible with the constructions of Hartl and Faltings above), by using $G$-bundles on the Fargues-Fontaine curve, cf. def. \ref{D:admissible qs}. In \cite{Har} sec. 9 and \cite{Ra2} A.20,  Hartl and Rapoport asked when we do have
\[ \breve{\Fc}(G,\mu, b)^{a}=\breve{\Fc}(G,\mu, b)^{wa} \ \ ?\]
 For $G=\GL_n$ Hartl gave a complete solution to this question (\cite{Har} theo. 9.3).
 \\
 
Let $b\in G(\breve{F})$ be an element such that its associated $\sigma$-conjugacy class \[[b]\in B(G, \mu)\] is the unique \textsl{basic} element in the Kottwitz set $B(G,\mu)$ (\cite{Kot2} sec. 6). In this paper we prove  the following theorem that was conjectured by Fargues and Rapoport (\cite{GoHeNi} conj. 0.1).

\begin{theorem}[theo. \ref{theo:main theorem}]
\label{conjecture} The equality
 $\breve{\Fc}(G,\mu, b)^{wa}=\breve{\Fc}(G,\mu, b)^{a}$ holds if and only if $B (G,\mu)$ is fully Hodge-Newton decomposable.
\end{theorem}

Recall that, roughly, the set $B(G,\mu)$ is fully Hodge-Newton decomposable if for any non basic element of $B(G,\mu)$ its Newton polygon, seen as an element of a positive Weyl chamber, touches the Hodge polygon defined by $\mu$ outside its extremities. 
 We refer to section \ref{sec:HN decomposability} and more precisely \ref{sec:sub sec HN decomposability} for the meaning of this notion. 
 In \cite{GoHeNi} theorem 2.5 there is a purely group theoretical classification of all the fully Hodge-Newton decomposable pairs $(G,\{\mu\})$, and in loc. cit. theo. 2.3 one can find further equivalent conditions for $(G,\{\mu\})$ being fully Hodge-Newton decomposable. In the following, Hodge-Newton will be usually abbreviated to HN for simplicity. 
\\

To prove the theorem, we make intensively use of the theory of $G$-bundles on the Fargues-Fontaine curve (\cite{F3}). More precisely, let $C|\breve{E}$ be a complete algebraically closed extension and $X_{C^\flat}$ be the Fargues-Fontaine curve attached to $C^\flat$ equipped with a closed point $\infty$ with residue field $C$. 
Let us recall that the main theorem of \cite{F3} says that 
\begin{align*}
B(G) & \iso H^1_{\et} (X,G) \\
[b'] & \longmapsto [\E_{b'}].
\end{align*}
 To each point $x\in \breve{\Fc}(G,\mu)(C)$, we can attach a modified $G$-bundle at $\infty$ of $\E_b$ $$\E_{b,x}$$
 on $X_{C^\flat}$.
 We assume that $G$ is quasi-split to simplify. Then we can use the notion of a semi-stable $G$-bundle and the theory of Harder-Narasimhan reduction on $X_{C^\flat}$ to study the modification $\E_{b,x}$ and the geometry of $\breve{\Fc}(G,\mu)$. The subspace $\breve{\Fc}(G,\mu, b)^{a}$ is defined as the locus where $\E_{b,x}$  is a semi-stable $G$-bundle.
The isomorphism class of $\E_{b,x}$ defines a stratification of $\breve{\Fc}(G,\mu)$ indexed by another Kottwitz set $B(J_b,\mu^{-1})$ (cf. sec. \ref{sec:HN stratification of the flag variety}).
We prove that this other Kottwitz set is fully HN decomposable if and only if $B (G,\mu)$ is fully HN decomposable (coro. \ref{coro:fully HN dec ssi idem}). 
To compare  $\breve{\Fc}(G,\mu, b)^{a}$ with the weakly admissible locus $\breve{\Fc}(G,\mu, b)^{wa}$, we also need to describe $\breve{\Fc}(G,\mu, b)^{wa}$ in terms of $G$-bundles on $X_{C^\flat}$. This is given in proposition \ref{P:wa}:  $x$ is weakly admissible if and only if for any parabolic reduction $(\E_{b,x})_P$ coming from a reduction of $b$ to the parabolic subgroup $P$, the usual numerical semi-stability condition is satisfied. 
With these ingredients at hand, we can prove the Fargues-Rapoport  conjecture by studying parabolic reductions of $\E_{b,x}$ and the one coming from a reduction of $b$.
\\

The reader who wants to have a feeling of how this type of proof works in a particular case can read the case of $SO(2,n-2)$ treated in \cite{F4} (see the appendix of \cite{Sh}). This case is instructive and  served as a starting point for the general proof of the implication that fully HN decomposability implies $\breve{\Fc}(G,\mu, b)^a=\breve{\Fc}(G,\mu, b)^{wa}$. It gives in particular the computation of the $p$-adic period space of K3 surfaces with supersingular reduction.
\\

We briefly describe the structure of this article. In section \ref{sec: G bundles on the curve}, we review the basic facts about $G$-bundles on the Fargues-Fontaine curve which we will use. In section \ref{sec:weak ad locus}, we review the definition of the weakly admissible locus $\breve{\Fc}(G, \mu, b)^{wa}$. In the quasi-split  case, we can give an equivalent definition for $\breve{\Fc}(G, \mu, b)^{wa}$ by using the theory of $G$-bundles on the Fargues-Fontaine curve. In section \ref{sec:ad locus}, we give the definition of the admissible locus $\breve{\Fc}(G, \mu, b)^{a}$ by using semi-stable $G$-bundles on the Fargues-Fontaine curve.  In section \ref{sec:HN decomposability}, we describe a generalized Kottwitz set for general groups, and we also discuss the fully HN decomposability condition and related properties that will be used in the sequel. In section \ref{sec:HN stratification of the flag variety}, we first explain the twin towers principle which is an important tool in the proof of the main theorem. Then we construct the Harder-Narasimhan stratification of the flag variety $\breve{\Fc}(G, \mu)$ and describe each stratum. With all these preparations, we finally prove the Fargues-Rapoport conjecture in section \ref{sec:proof of the main theorem}. Finally, in section \ref{sec:asymptotic}, we discuss the asymptotic geometry of period spaces. We introduce in particular a new conjecture (\ref{conj:ad egal was ssi domaine fondamental compact}) saying that $\breve{\Fc}(G,\mu,b)^{a}=\breve{\Fc}(G,\mu,b)^{wa}$ if and only if there exists a quasicompact fundamental domain for the action of $J_b(F)$ on $\breve{\Fc}(G,\mu,b)^a$. \\
 \\
\textbf{Acknowledgments.}  We would like to thank Sian Nie and Ulrich G\"ortz sincerely for valuable help in group theory.

\section*{Notations}

We use the following notations:
\begin{itemize}
\item $F$ is a finite degree extension of $\Q_p$ with residue field $\F_q$.
\item $\Fbar$ is an algebraic closure of $F$ and $\Gamma = \Gal (\Fbar |F)$.
\item $\breve{F} =\widehat{F^{un}}$ is the completion of the maximal unramified extension $\Fun$ of $F$ with Frobenius $\s$.
\item $G$ is a connected reductive group over $F$.
\item $H$ is a quasi-split inner form of $G$ equipped with an inner twisting $G_{\Fbar}\xrightarrow{\sim} H_{\Fbar}$.
\item $\boldsymbol{A}\subset \boldsymbol{T} \subset \boldsymbol{B}$ are a maximal unramified torus, $\boldsymbol{T}=C_H (\boldsymbol{A})$ a minimal Levi and $\boldsymbol{B}$ a Borel subgroup in $H$. We reserve the notation $T$ for a maximal torus in $G$.
\item $(X^* (\TT),\Phi, X_*(\TT), \Phic )$ is the absolute root datum with positive roots $\Phi^+$ and simple roots $\Delta $ with respect to the choice of $\boldsymbol{B}$.
\item $(X^*(\AA), \Phi_0,  X_*(\AA),\Phic_0)$ is the relative root datum with positive roots $\Phi^+_0$ and simple (reduced) roots $\Delta_0$.
\item If $M$ is a standard Levi subgroup in $H$ we note by a subscript $M$, for example $\Phi_M$,
 the corresponding roots or coroots showing up in $\Lie\, M$. For example $M\mapsto \Delta_{0, M}$ induces a bijection between the standard Levi subgroups and subsets of $\Delta_0$.
 \item If $\mathbb{D}$ is the slope pro-torus with characters $X^* (\mathbb{D})=\Q$,
 we set $$\mathcal{N} (G)= \big [ \Hom (\mathbb{D}_{\overline{F}},G_{\overline{F}}) \, /\, G(\overline{F})\text{-conjugacy}\big ]^\Gamma,  $$
 the Newton chamber. Via the inner twisting between $G$ and $H$, there is an identification 
 $$
 \mathcal{N}(G)=\mathcal{N} (H)=X^* (\AA)_\Q^+.
 $$
 This is equipped with the usual order $\nu_1\leq \nu_2$ if and only if $\nu_2-\nu_1 \in \Q_{\geq 0}\Phi_0^+$.
 \item $\pi_1(H)=X_\ast(\boldsymbol{T})/\lan \Phi^\vee \ran$ is the algebraic fundamental group of $H$, and $\pi_1(H)_\Gamma$ is its Galois coinvariant. Via the inner twisting between $G$ and $H$, there are identifications
 \[\pi_1(G)=\pi_1(H),\quad \pi_1(G)_\Gamma=\pi_1(H)_\Gamma.\] 
 \item $G_{ad}$ is the adjoint group associated to $G$, $G_{der}\subset G$ is the derived subgroup, and $G_{sc}\ra G_{der}$ is the simply connected cover of $G_{der}$.
\end{itemize}

\section{$G$-bundles on the Fargues-Fontaine curve}
\label{sec: G bundles on the curve}

In this section, we review some basic facts about the Fargues-Fontaine curve and $G$-bundles on it, cf. \cite{FF, F3}. This theory will be the basic tool for our study of $p$-adic period domains.
We change slightly the notations from \cite{FF} and \cite{F3} to be in accordance with \cite{Ra2}.

\subsection{The Fargues-Fontaine curve}
The Fargues-Fontaine curve $X$ over $F$ is associated to the choice of a characteristic $p$ perfectoid field $K|\F_q$.
 We note $\pi_F$ for a uniformising element of $F$. 
It has several incarnations.
\subsubsection{The adic curve}
The adic curve $X^{ad}$ admits the following adic uniformization \[X^{ad}=Y/\varphi^\Z,\]where $Y=\Spa(W_{\Ol_F}(\Ol_K))\setminus V(\pi_F[\varpi_K])$, with $\varpi_K\in K$ satisfying $0<|\varpi_K|<1$. The action of the Frobenius $\varphi$ on the ramified Witt vectors is given by
\[\varphi \Big (\sum_{n\geq 0} [x_n]\pi_F^n\Big )=\sum_{n\geq 0} [x_n^q]\pi_F^n.\]
It induces a totally discontinuous action on $Y$ without fixed point.

\subsubsection{The algebraic curve}
There is a natural line bundle $\Ol(1)$ on $X^{ad}$, corresponding to the $\varphi$-equivariant line bundle on $Y$ whose underlying line bundle is trivial and for which the Frobenius  is $\pi_F^{-1}\varphi$. Set $\Ol(n)=\Ol(1)^{\otimes n}$, and \[P=\bigoplus_{n\geq 0}H^0(X^{ad},\Ol(n)).\]We have
\[H^0(X^{ad}, \Ol(n))=\Ol(Y)^{\varphi=\pi_F^n}.\]Let \[X=\Proj(P).\] By \cite{FF}, this is a one dimensional noetherian regular scheme over $F$.
There exists a morphism of ringed spaces
\[X^{ad}\lra X,\]and $X^{ad}$ may be viewed as the analytification of $X$ in some generalized sense. Typically there is a natural subset $|X^{ad}|^{cl}\subset |X^{ad}|$ of "classical Tate points" and the preceding induces a bijection
$$
|X^{ad}|^{cl} \xrightarrow{ \ \sim\ } |X|,
$$
where $|X|$ denotes the closed points of $X$.

Let $F$ be fixed. 
Suppose that instead of beginning with the characteristic $p$ datum $K|\mathbb{F}_q$ we start with $K^\sharp|F$ a perfectoid field and set $K=K^{\sharp,\flat}$. 
Then, the curve $X$ is equipped with a closed point $\infty$ with residue field $k(\infty)=K^\sharp$ and $\widehat{\Ol}_{X,\infty}=B^+_{dR} ( K^\sharp)$, cf. \cite{FF}. 
Quickly in the following $K$ will be supposed to be algebraically closed (we will test the semi-stability of some vector bundles by "specializing at a geometric point"). This is equivalent to $K^\sharp$ being algebraically closed.

\subsection{$G$-bundles}
Let $\Bun_{X}$ and $\Bun_{X^{ad}}$ be the categories of vector bundles on $X$ and $X^{ad}$ respectively. The morphism $X^{ad}\ra X$ induces a GAGA functor
\[\Bun_{X}\lra \Bun_{X^{ad}}.\]
\begin{theorem}[\cite{ F2, KL}]
	The GAGA functor induces an equivalence of categories \[\Bun_{X}\st{\sim}{\lra} \Bun_{X^{ad}}.\]
\end{theorem}

{\it We assume from now on that $K$ is algebraically closed.} For example, $K=C^\flat$ with $C|\Fbar$  a complete algebraically closed field.
Let $$\phmod_{\Fb}$$ be the category of isocrystals relative to $\Fb|F$.
 For any $(D,\varphi)\in \phmod_{\Fb}$, we can construct a vector bundle $\E(D,\varphi)$ on $X$ 
 associated to the graded $P$-module 
$$
\bigoplus_{n\geq 0}\big (D\otimes_{\Fb} \Ol(Y)\big )^{\varphi\otimes \varphi=\pi_F^n}.
$$
Via GAGA this corresponds to the vector bundle $Y\underset{\varphi^\Z}{\times} D$ on $X^{ad}=Y/\varphi^\Z$.

\begin{theorem}[\cite{FF}]
	The functor $\E(-): \phmod_{\Fb}\lra \Bun_{X}$ is essentially surjective.
\end{theorem}
We will need the following fact:
\begin{theorem}[\cite{FF}]
The degree map of a line bundle induces an isomorphism
\[\deg: \tr{Pic}(X)\st{\sim}{\lra} \Z.\]
\end{theorem}

Note that the fact the degree of a line bundle is well defined is not evident. This is a consequence of the fact that the curve is "complete": the degree of a principal divisor is zero. This is essential to develop a theory of Harder-Narasimhan reduction.
\\

We have the following two equivalent definitions of a $G$-bundle on $X$:
\begin{enumerate}
	\item an exact tensor functor $\Rep \, G\ra \Bun_X$, where $\Rep \, G$ is the category of rational algebraic representations of $G$,
	\item a $G$-torsor on $X$ locally trivial for the \'etale topology.
\end{enumerate}
Recall that an isocrystal with $G$-structure is an exact tensor functor \[\Rep\, G\lra \phmod_{\Fb}.\]
 Let $B(G)$ be the set of $\sigma$-conjugacy classes in $G(\breve{F})$ (\cite{Kot1, Kot2, RR}).
If $b\in G(\breve{F})$, it  defines an isocystal with $G$-structure \[\begin{split} \mathcal{F}_b: \Rep\, G&\lra \phmod_{\breve{F}} \\ V&\longmapsto (V_{\Fb}, b\sigma).\end{split}\] 
Its isomorphism class  depends only on the $\sigma$-conjugacy class $[b]\in B(G)$ of $b$. Conversely, any isocrystal with $G$-structure arises in this way. Thus $B(G)$ is the set of isomorphism classes of isocrystals with $G$-structure (\cite{RR} rem. 3.4 (i)). For $b\in G(\breve{F})$, let $\E_b$ be the composition of the above functor $\mathcal{F}_b$ and \[\E(-):  \phmod_{\breve{F}}\lra  \Bun_X.\] 
In this way, the set $B(G)$ also classifies $G$-bundles on $X$. In fact, we have
\begin{theorem}[\cite{F3},\cite{Ans}]\label{T: G-bundles}
	There is a bijection of pointed sets
	\[\begin{split} B(G)&\st{\sim}{\lra} H^1_{\textrm{\'et}}(X, G) \\
	[b]&\longmapsto [\E_b]. \end{split}	\]
\end{theorem}

\subsection{The Harder-Narasimhan reduction in the quasi-split case}
We assume that $G=H$ is quasi-split in this subsection.
The theory of Harder-Narasimhan reduction (\cite{BH} for example) applies for $G$-bundles over the Fargues-Fontaine curve $X$ (\cite{F3} sec. 5.1).
 If $G'\subset G$, then a reduction of a $G$-bundle $\E$ to $G'$ is a $G'$-bundle $\E_{G'}$ together with an isomorphism \[\E_{G'}\times_{G'} G\iso \E.\] 
Recall the following definition of        a semi-stable $G$-bundle.
\begin{definition}\label{D:semistable}
Let $\E$ be a $G$-bundle on $X$. It is called semi-stable if for any standard parabolic subgroup $P$ of $G$, any reduction $\E_P$ of $\E$ to $P$, and any $\chi\in X^\ast(P/Z_G)^+$, we have \[\deg\, \chi_\ast\E_P\leq 0.\]
\end{definition}
Let $M\subset P$ be the associated standard Levi subgroup of $G$, which we will identify with $P/R_uP$.
Then as $X^\ast(P)=X^\ast(M)$, we have \[\chi_\ast\E_P=\chi_\ast ( \E_P \times_P M).\] 
One proves, as usual, that  $\E$ is semistable if and only if the associated adjoint bundle \[\mathrm{Ad}(\E):=\E\times_{G, \mathrm{Ad}}\Lie\, G\] is semi-stable. We will later use  the following well-known criterion.

\begin{lemma}\label{lemma:reduction to the case of maximal parabolic}
The following are equivalent:
\begin{enumerate}
\item 
The $G$-bundle $\E$ on $X$ is semi-stable.
\item For any standard parabolic subgroup $P$ and any reduction $\E_P$ to $P$, one has 
$$
\deg ( \E_P\times_{P, \mathrm{Ad}} \Lie \, G / \Lie \, P ) \geq 0.
$$
\item The same holds for any maximal standard parabolic subgroup.
\end{enumerate}
\end{lemma}
\begin{proof}
This is a consequence of the fact that for any standard parabolic subgroup $P$, via the adjoint representation, $\det ( \Lie \, G / \Lie \, P)^{-1}\in X^*(\TT)^+$, and moreover when $P$ goes through the set of maximal parabolic subgroups those are $>0$ multiples of the fundamental weights. More precisely, if $M$ is the standard Levi attached to $P$ and  $\rho_M =\frac{1}{2} \sum_{\alpha\in \Phi (\TT)_M^+}\alpha$ then for $\beta \in \Delta_M$, $\langle \beta^\vee, \rho_M\rangle =1$ and for $\beta\in \Delta \setminus \Delta_M$, $\langle \beta^\vee,\rho_M\rangle \leq 0$. One concludes using that $\det ( \Lie \, G / \Lie \, P)^{-1}=2 \rho -2 \rho_M$.
\end{proof}

One can rephrase  this lemma in a more geometric way. In fact, if the reduction $\E_P$ corresponds to the section $s$ of $P\backslash \E\rightarrow X$, then 
$$
\E_P\times_{P, \mathrm{Ad}} \Lie \, G / \Lie \, P = s^* T (P\backslash \E) 
$$
is the pullback of the tangent bundle.
\\

For a general $G$-bundle, as usual, we have the following theorem.
\begin{theorem}\label{T:HN reduction}
Let $\E$ be a $G$-bundle on $X$. Then there exists a unique standard parabolic subgroup $P$ of $G$ and a unique reduction $\E_P$ to $P$, such that \begin{enumerate}
	\item the associated $M$-bundle $\E_P\times_{P}M$ is semi-stable,
	\item  for any $\chi\in X^\ast(P/Z_G)\setminus \{0\}\cap \mathbb{N}.\Delta $, we have $\deg\, \chi_\ast\E_P> 0$.
\end{enumerate}
\end{theorem}
\begin{proof}
See \cite{F3} 5.1, where one can apply the arguments of \cite{BH}.
\end{proof}
The reduction $\E_P$ in the above theorem is called the Harder-Narasimhan reduction of $\E$.
Let $\E$ be a $G$-bundle on $X$  with Harder-Narasimhan reduction $\E_P$. 
Then we get an element \[\nu_{\E}\in X_\ast(\AA)_\Q=X_\ast(\TT)_\Q^{\Gamma} \]
 by the Galois invariant morphism \[\begin{split}
X^\ast(P)&\lra\Z\\ \chi &\longmapsto \deg\,\chi_\ast\E_P
\end{split}\]
and the inclusion 
$$
\Hom_\Z ( X^* (P),\Z) \subset X_*(\TT)_\Q.
$$
In fact, we have $\nu_{\E}\in  X_\ast(\AA)_\Q^+$,  and moreover $M$ is the centralizer of $\nu_\E$. 
As in the classical theory of $G$-bundles on curves, the vector  $\nu_{\E}$ is called the Harder-Narasimhan polygon of $\E$ as an element
$$
\nu_\E \in \mathcal{N}(G)=X_\ast(\AA)_\Q^+.
$$

Later we will need the following. Recall that if $\E$ is a vector bundle on $X$ with Harder-Narasimhan filtration $(0)=\E_0\subsetneq \E_1\subsetneq \dots \subsetneq \E_r=\E$ then:
\begin{enumerate}
\item 
 for any subvector bundle $\Fc\subset \E$ the point $(\mathrm{rk}\, \E,\deg\, \E)$ lies under the Harder-Narasimhan polygon of $\E$,
 \item  this point lies on this polygon then there exists an index $i$ such that $\E_i\subset \Fc\subset \E_{i+1}$.
\end{enumerate}
Here is the generalization we will need.

\begin{theorem}[\cite{Schi} theo. 4.5.1]
\label{theoSchi}
For $\E$ a $G$-bundle on $X$ equipped with a reduction $\E_Q$ to the standard parabolic subgroup $Q$ consider the vector 
\begin{align*}
v: X^*(Q) & \longrightarrow \Z \\
\chi & \longmapsto \deg \, \chi_* \E_Q
\end{align*}
seen as an element of $X_*(\AA)_\Q$.  Then
\begin{enumerate}
\item One has $v\leq \nu_\E$.
\item If this inequality is moreover an equality and $\E_P$ is the canonical reduction of $\E$ then $Q\subset P$ and $\E_P \simeq \E_Q\times_Q P$. 
\end{enumerate}
\end{theorem}

As a corollary the vector $\nu_\E$ can be defined as being the supremum of all such vectors $v$ associated to all possible reductions $\E_Q$ in the poset $X_* (\AA)_\Q$.

\subsection{The Harder-Narasimhan polygon in general and the first Chern class}
\label{sec:HN polygon in general}

Suppose now that $G$ is not necessarily quasi-split. Recall that the inner twisting $G_{\Fbar}\xrightarrow{\sim} H_{\Fbar}$ induces an identification 
\begin{align*}
\mathcal{N} (G) \iso \mathcal{N}(H)  
= X_*(\AA)_\Q^+.
\end{align*}

There is a canonical $\DD$-torsor $\mathscr{T}$ on $X\otimes_F \Fun$. More precisely, if $F_h|F$ is the degree $h$ unramified extension, we have
$$
X_F\otimes_F F_h =X_{F_h},
$$
the curve attached to $F_h$.
If $\pi_h:X_{F_h}\rightarrow X_F$ is the natural projection,
one has a canonical identification 
$$
\pi_h^* \Ol_{X_F} (1) = \Ol_{X_{F_h}} (1)^{\otimes h}.
$$
The compatible system of $\G_m$-torsors attached to $\big ( \Ol (X_{F_h})(1) \big )_{h\geq 1}$  defines $\mathscr{T}$.
As a consequence of the classification theorem of \cite{F3}, any $G$-torsor on $X_{\Fun}$ has a reduction to a torus and we have

\begin{theorem}
The pushforward of the $\mathbb{D}$-torsor $\mathscr{T}$ induces a bijection 
$$
\Hom ( \mathbb{D}_{\Fun}, G_{\Fun} ) / G(\Fun) \iso H^1_{\text{\'et}} (X_{\Fun}, G).
$$
\end{theorem}

Together with the pullback from $X$ to $X_{\Fun}$, this defines a map
\begin{align*}
H^1_{\text{\'et}} (X, G) & \longrightarrow  \mathcal{N}(G) \\
 [\E ]  & \longmapsto \nu_{\E}.
\end{align*}
One can moreover define the $G$-equivariant first Chern class of $\E$ as a map (\cite{F3})
$$
c_1^G : H^1_{\text{\'et}} (X,G)\longrightarrow \pi_1 (G)_\Gamma.
$$
This generalizes the degree of a vector bundle for $G=\GL_n$. 
The quickest way to define it is through abelianized cohomology in the topos $X_{\et}$ (\cite{Breen}) via the map
$$
H^1_{\et} ( X, G) \longrightarrow H^1_{\et} ( X , [G_{sc}\rightarrow G]),
$$
the homotopy equivalence $[T_{sc}\rightarrow T] \rightarrow [G_{sc}\rightarrow G]$ for a maximal torus $T$ in $G$, and the canonical isomorphism $X_* (S)_\Gamma \xrightarrow{\sim} H^1_{\et} ( X,S)$ for a torus $S$ (\cite{F3} theo. 2.8) (see the next subsection for more explanations in the case of $B(G)$ for this type of abelianization construction).

\subsection{Newton map and Kottwitz map}\label{Newton and Kottwitz map}

We keep the same notations. 
 The set $B(G)$ of $\sigma$-conjugacy classes in $G(\Fb)$ can be described by two invariants. One invariant is the Newton map (\cite{Kot1} sec. 4).
For each $b\in G(\Fb)$ on can attach a slope morphism 
$$
\nu_b:\DD_{\Fb}\longrightarrow G_{\Fb}.
$$
 Up to $\s$-conjugating $b$ one can suppose that it is defined over $\Fun$ (the Tannakian category of isocrystals has a fiber functor over $\Fun$). Moreover, since $\nu_b^\s = b^{-1} \nu_b b$, its conjugacy class is defined over $F$. We thus obtain an application
 \begin{align*}
 \nu: B(G) & \longrightarrow \mathcal{N}(G) \\
 [b] & \longmapsto [\nu_b].
 \end{align*} 
The other invariant is the Kottwitz map (\cite{RR}  1.15, \cite{Kot2} 4.9 and 7.5) 
\begin{align*} \kappa: B(G)&\lra \pi_1(G)_\Gamma\\
[b]&\longmapsto \kappa([b]).\end{align*} 
In the following we will simply write $\kappa([b])$ as $\kappa(b)$.
The quickest way to define the Kottwitz map is via the abelianization of Kottwitz set \`a la Borovo\"i. More precisely, if we set 
$$
B_{ab} (G) := H^1( \s^\Z, [G_{sc} (\Fb) \drt G (\Fb)] )
$$
(cohomology with coefficient in a crossed module where $G(\Fb)$ is in degree $0$) there is an abelianization map
$$
B(G)\longrightarrow B_{ab} (G)
$$
induced by $[ 1\rightarrow G(\Fb)] \rightarrow [G_{sc} ( \Fb)\rightarrow G(\Fb) ]$. Moreover, if $T$ is a maximal torus in $G$ then $[T_{sc}\rightarrow T]\rightarrow [G_{sc}\rightarrow G]$ is an homotopy equivalence and this induces an isomorphism 
$$
H^1 ( \s^\Z , [T_{sc}(\Fb) \rightarrow T(\Fb)] ) \iso B_{ab}(G).
$$
Now the left member is identified with the cokernel of 
$$
B(T_{sc})\longrightarrow B (T)
$$
which, according to Kottwitz, is the same as the cokernel of 
$$
X_*(T_{sc})_\Gamma\longrightarrow X_* (T)_\Gamma
$$
that is to say $\pi_1(G)_\Gamma$. This is canonically defined independently of the choice of $T$ since the Weyl group of $T$ acts trivially on the cokernel of $X_* (T_{sc})\rightarrow X_* (T)$.
\\

The induced map
\[(\nu, \kappa): B(G)\lra \mathcal{N}(G) \times \pi_1(G)_\Gamma\] is injective (\cite{Kot2} 4.13). Let $B(G)_{basic}\subset B(G)$ be the subset of basic elements. Then the restriction of $\kappa$ to $B(G)_{basic}$ induces a bijection
\[\kappa: B(G)_{basic}\st{\sim}{\lra} \pi_1(G)_\Gamma \]
(\cite{Kot1} prop. 5.6, \cite{RR} theo. 1.15).
In view of Theorem \ref{T: G-bundles}, the two maps $\nu$ and $\kappa$ have the following geometric interpretations in terms of $G$-bundles.
\begin{theorem}[\cite{F3} prop. 6.6 and prop. 8.1]\label{T:NK}
We have
\begin{enumerate}
	\item $\nu_{\E_b}=w_0(-[\nu_b])$, where $w_0$ is the element of longest length in the Weyl group acting on $X_* (\AA)_\Q$.
	\item $ c_1^G ( \E_b)=- \kappa (b)$.
\end{enumerate}
\end{theorem}
A particular case of the above (1) says that {\it $\E$ is semi-stable if and only if $[b_\E]\in B(G)$ is basic.}

\begin{example}
If  $G_{der}$is  simply connected one has
$$
\begin{tikzcd}
\{G\text{-bundles} / X \}/\sim  \ar[rr,"\text{push forward}"] \ar[rrr, bend right=15,"c_1^G"'] 
 && \{ G/G_{der}\text{-bundles} / X\}/\sim \ar[r, "\sim"]  &\pi_1 (G)_\Gamma 
\end{tikzcd}
$$
where, when $G$ is quasi-split, the first map is a bijection when restricted to semi-stable bundles. 
\end{example}

%

\section{The weakly admissible locus}
\label{sec:weak ad locus}

\subsection{Background}

Let  $\{\mu\}$ be a geometric conjugacy class of a cocharacter $\mu: \G_m\ra G_{\Fbar}$. 
Unless clearly stated otherwise  $\mu$ is supposed to be minuscule.
The pair $(G, \{\mu\})$ will be fixed in the rest of this paper.  We will sometimes see $\{\mu\}$ as an element of $X_*(\TT)^+$ that we again denote $\mu$ by abuse of notation.
  We get the following associated objects:
\begin{itemize}
	\item the local reflex field $E=E(G, \{\mu\})$, a finite extension of $F$ inside $\Fbar$, which is the field of definition of $\{\mu\}$,
	\item the flag variety $\Fc(G,\mu)$ over $\breve{E}$, where $\breve{E}$ is the completion of the maximal unramified extension of $E$ (in fact $\Fc(G,\mu)$ is defined over $E$, but we will not need this fact),
	\item and the Kottwitz set \[B(G,\mu)=\{[b]\in B(G)\;| \; [\nu_{b}]\leq \mu^\diamond, \kappa(b)=\mu^\sharp\},\] which is a finite subset of $B(G)$; where as usual \[\mu^\diamond=[\Gamma:\Gamma_\mu]^{-1}\sum_{\tau\in \Gamma/\Gamma_\mu}\mu^\tau \in X_* (\AA)_\Q^+=\big (X_*(\TT)^{+}_\Q)^{\Gamma},\] and $\mu^\sharp\in \pi_1(G)_\Gamma$ is the image of $\mu$ via $\pi_1 (G)=X_*(\TT)/X_*(\TT_{sc})$.
\end{itemize}

In the preceding definition of $B(G,\mu)$, the condition $[\nu_b]\leq \mu^\diamond$ implies that $\kappa (b)-\mu^{\sharp} \in \pi_1 (G)_{\Gamma,tor}= H^1(F,G)$. The condition $\kappa (b)=\mu^\sharp$ requires that this cohomology class is trivial.
\\

In the following we will also denote by $\Fc(G,\mu)$ the associated adic space over $\Spa(\breve{E})$.
Fix an element $$[b]\in B(G).$$  
We get
the reductive group $J_b$ over $F$, the $\s$-centralizer of $b$,
 which only depends on $[b]$ up to isomorphism. For any $F$-algebra $R$, we have \[J_b(R)=\{g\in G(\breve{F}\otimes_F R)|\; gb\sigma(g)^{-1}=b\}.\]
The  group $J_b(F)$ acts naturally on the flag variety $\Fc(G,\mu)$ over $\breve{E}$ via $J_b(F)\subset G(\Fb)$.
For the triple $(G,\{\mu\}, [b])$ as above,
let us recall the definition of the weakly admissible locus (\cite{RZ} chap. 1)
 \[\Fc(G,\mu, b)^{wa}\subset \Fc(G,\mu).\]

First we recall the definition of a weakly admissible filtered isocrystal. Let $K|\Fb$ be a complete field extension. A filtered isocrystal $\V=(V, \varphi, \Fil^\bullet V_K)$ is called weakly admissible if for any subobject $\V'=(V', \varphi, \Fil^\bullet V_K')$ of $\V$, with $V'$ a $\ph$-stable $\Fb$-subspace of $V$ and $\Fil^\bullet V_K'=V_K'\cap \Fil^\bullet V_K$, we have
\[t_H(\V)=t_N(\V) \quad \tr{and}\quad t_H(\V')\leq t_N(\V')\] where $t_N(\V)=v_{\pi_F} (\det \varphi)$ is the $\pi_F$-adic valuation of $\det \ph$, and \[t_H(\V)=\sum_{i\in\Z}i\cdot \dim_K gr^i_{\Fil^\bullet}(V_K).\] If we define the slope of $\V$  as \[\mu(\V)=\frac{t_H(\V)-t_N(\V)}{\dim V},\]then $\V$ is weakly admissible if and only if it is semi-stable (for the slope function $\mu$) and $\mu(\V)=0$.

 If $K|\Fb$ is a finite extension, then  the category of weakly admissible filtered isocrystals is equivalent to the category of crystalline representations of $\Gal(\ov{K}|K)$ with coefficients in $F$ whose Sen operator is $F$-linear (\cite{CoFo}), that is to say weakly admissible is equivalent to admissible. We refer to chapter 
10 of \cite{FF} for a proof of this result using the curve (this proof was a huge inspiration to study modifications of vector bundles on the curve).
\\

Let us come back to the general setting. Let $\mu'\in \{\mu\}$ be defined over $K|\Fb$.
There is a functor \[\begin{split} \I_{b,\mu'}: \Rep\, G&\lra \varphi\mathrm{-}\Fil\Mod_{K/\Fb} \\ (V,\rho) &\longmapsto (V_{\Fb}, \rho(b)\sigma, \Fil_{\rho\circ\mu'}^\bullet V_K).\end{split}\]
Let us recall that the category of weakly admissible filtered isocrystals is Tannakian.

\begin{definition}[\cite{RZ} def. 1.18]\label{D:wa}
We call the pair $(b,\mu')$ weakly admissible if one of the following equivalent conditions is satisfied:
\begin{enumerate}
	\item for any $(V,\rho)\in \Rep \, G$, the filtered isocrystal $\I_{b,\mu'}(V,\rho)$ is weakly admissible;
	\item there is a faithful representation $(V,\rho)$ of $G$ such that $\I_{b,\mu'}(V,\rho)$ is weakly admissible.
\end{enumerate}
\end{definition}
It is also equivalent to the condition that its image in $G/G_{der}$ is weakly admissible (i.e. $
[b]\in A(G,\mu)$, cf. prop. 2.2) and 
$\I_{b,\mu'}(\Lie\, G)$ is weakly admissible for the adjoint representation $\Lie\, G$ of $G$, cf. \cite{DOR} def. 9.2.14 and coro. 9.2.26. 
We will give an equivalent geometric definition of weakly admissible in terms of $G$-bundles on the Fargues-Fontaine curve later. 
\\

Now for any point $x\in \Fc (G,\mu)(K)$ we have the associated filtration $\Fil_x$ of $\Rep\, G$ coming from a cocharacter $\mu_x\in \{\mu\}$ defined over $K$. Then $x$ is called weakly admissible if the pair $(b, \mu_x)$ is weakly admissible in the above sense. By proposition 1.36 of \cite{RZ}, this defines a partially proper open subspace \[\Fc(G,\mu, b)^{wa}\subset \Fc(G,\mu),\] such that $\Fc(G,\mu, b)^{wa}(K)$ is the set of weakly admissible points in $ \Fc(G,\mu)(K)$. It is of the form 
$$
\Fc(G,\mu) \setminus \bigcup_{i\in I} J_b (F). Z_i
$$
where $(Z_i)_{i\in I}$ is a finite collection of Zariski closed Schubert varieties. The action of $J_b(F)$ on $\Fc(G,\mu)$ stabilizes the subspace $\Fc(G,\mu,b)^{wa}$. 
\\

Recall the following basic fact.
\begin{proposition}[\cite{RV} prop. 3.1]\label{P:wa non empty}
The open subset 
$\Fc(G,\mu,b)^{wa}$ is non empty if and only if $[b]\in A(G, \mu):=\{[b]\in B(G)\;| \;\ [\nu_{b}]\leq \mu^\diamond \}$.
\end{proposition}
As $B(G, \mu)\subset A(G,\mu)$ and we will be interested in the case $[b]\in B(G, \mu)$, our $\Fc(G,\mu,b)^{wa}$ will be non empty.

\subsection{Weak admissiblity and the curve}

The preceding definition of a weakly admissible point is Tannakian. 
We now give a geometric weak admissibility criterion in terms of the curve when $G$ is quasi-split. 
Let $C|\Fbar$ be algebraically closed complete and consider the curve attached to $C^\flat$ equipped with its point $\infty \in |X|$ with residue field $C$ and $$\BpdR:=\BpdR (C) = \widehat{\Ol}_{X,\infty}.$$ Consider the $C$-points of the $\BdR$-affine Grassmanian attached to $G$ (\cite{Sch}, \cite{FS},
 we only need its $C$-points, not the geometric diamond structure on it for what we do)
$$
\Gr^{\mathrm{B_{dR}}}_G (C):= G ( \BdR )/ G ( \BpdR).
$$
Since we only consider the $C$-points of $\Gr^{\mathrm{B_{dR}}}_G$, the reader who believes in Zorn's lemma can fix an isomorphism $C\llbracket t\rrbracket\xrightarrow{\sim} \BpdR$, that is to say a section of $\theta:\BpdR\rightarrow C$. After such a choice this is reduced to the $C$-points of the "classical" affine Grassmian. Recall nevertheless that there is a canonical section over $\Fbar$ and thus $\Fbar \subset \BpdR$ canonically, in particular we can define $\mu(t)\in G(\BpdR)$.
\\

Choose $b\in G(\Fb)$ that gives us the $G$-bundle $\E_b$. Its pullback via $\Spec (\BpdR)\rightarrow X$ is canonically trivialized. 
For each $x\in \Gr^{\mathrm{B_{dR}}}_G (C)$ one can construct a modification 
$$
\E_{b,x}
$$
of $\E_b$ \`a la Beauville-Laszlo (\cite{CS}  3.4.5, \cite{F2} 4.2, \cite{F} 3.20). This is given by gluing $\E_{b | X\setminus \{\infty\}}$ and the trivial $G$-bundle on $\Spec ( \BpdR)$ via the gluing datum given by $x$. 
\\

For $\mu\in X_*(\TT)^+$  the corresponding affine Schubert cell is
$$
\Gr^{\mathrm{B_{dR}}}_{G,\mu} (C) = G ( \BpdR)\mu (t)^{-1}  G ( \BpdR) / G(\BpdR) \subset \Gr^{\mathrm{B_{dR}}}_G (C).
$$
For $x\in \Gr^{\mathrm{B_{dR}}}_{G,\mu} (C)$ one has (see \cite{CS} lemma 3.5.5)
\begin{eqnarray*}\label{eq:formule permutation Hecke composantes}
c_1^G ( \E_{b,x} )& =& \mu^{\sharp} + c_1^G ( \E_b ) \\
&=&  \mu^{\sharp} - \kappa (b).
\end{eqnarray*}

\begin{remark}
In terms of the stack $\Bun_G$ (\cite{F}, \cite{FS})
the preceding formula gives the way the Hecke correspondence
\begin{tikzcd}[column sep=0.6cm]
\mathrm{Hecke}_\mu \ar[shift left=1.1mm, r] \ar[shift right=1.1mm, r] & \Bun_G
\end{tikzcd}
 permute the components of $\Bun_G= \coprod_{\alpha\in \pi_1 (G)_\Gamma} \Bun_G^{\alpha}$ where $\Bun_G^\alpha$ is the open/closed substack where $c_1^G=\a$. It says that 
$$
\mathrm{Hecke}_\mu ( \Bun_G^\a )
=\Bun_G^{\a + \mu^\sharp}.
$$
\end{remark}

Recall that for any $\mu$
 we have the {\it Bialynicki-Birula map } (\cite{CS} prop. 3.4.3 in general,  we don't need its diamond version but just its evaluation on $C$-points)
$$
\pi_{G,\mu}:\Gr^{\mathrm{B_{dR}}}_{G,\mu} (C) \longrightarrow \Fc  (G,\mu) (C).
$$
{\it When $\mu$ is minuscule this is an isomorphism} induced by applying  $\theta$ via
$$
G(\BpdR) \cap \mu(t)^{-1} G(\BpdR) \mu(t) = \{ g\in G(\BpdR)\ |\ \theta (g) \in P_\mu (C)\}, 
$$
a parahoric subgroup in $G(\BpdR)$.
Let us recall the following well-known lemma that is deduced from the properness of $G/P$ together with the fact that $X$ is a Dedekind scheme.

\begin{lemma}\label{P:redction modfication}
Let $\E$ and $\E'$ be two $G$-bundles on $X$ with a modification $\E|_{X\setminus\{\infty\}}\st{\sim}{\ra}\E'|_{X\setminus\{\infty\}}$. Then for any parabolic subgroup $P$ of $G$, we have a bijection
\[ \{\tr{Reductions of}\ \E \tr{ to }P\}\lra \{ \tr{Reductions of}\ \E'\tr{ to }P\}.\]
\end{lemma}
We will need the following key definition. 
\begin{definition}
Let $b\in G(\Fb)$ be an element. For a Levi subgroup $M$  of $G$, a reduction of $b$ to $M$ is an element $b_M\in M(\Fb)$ together with an element $g\in G(\Fb)$, such that $b=gb_M\sigma(g)^{-1}$. Such a reduction $(b_M,g)$ is considered to be equivalent to $(h b_M h^{-\s}, gh^{-1})$ for any $h\in M(\Fb)$.  
We use the same notation and terminology for parabolic subgroups.
\end{definition}

The reductions of $b$ to $M$ and the reductions of $\E_b$ to $M$ are essentially the same.  However, for the reductions to parabolic subgroups, the situation is different. Any reduction $b_P$ of $b$ to a parabolic subgroup $P$ induces a reduction $\E_{b_P}$ of $\E_b$ to $P$. {\it But the converse is false and this is the main reason why in general  the weakly admissibility and the admissibility conditions differ.} Let us notice that if $M$ is a Levi subgroup of the parabolic subgroup $P$
(semi-simplicity of the category of isocrystals)
$$
B(M)\iso B(P).
$$
Thus, reductions of $b$ to $M$ or $P$ are essentially the same.
\\

 Suppose now that $\mu$ is minuscule, $b_M$ is a reduction of $b$ to $M$ with associated reduction $b_P$ to $P$. For $x\in \Fc  (G,\mu)(C)$ we deduce a reduction $(\E_{b,x})_{P}$ of $\E_{b,x}$ via lemma \ref{P:redction modfication}. This is  a modification of $\E_{b_P}$. Now, we have the decomposition in Schubert cells
of $\Fc  (G,\mu)(C)=G(C)/P_\mu (C)$ according to the $P$-orbits
$$
\Fc  (G,\mu) (C) = \coprod_{w\in W_P\backslash W / W_{P_\mu}} 
\Fc  (G,\mu) (C)^w
$$
where 
\begin{align*}
\Fc  (G,\mu) (C)^w
&=
P(C) \overset{.}{w} P_{\mu} (C) /P_{\mu}(C) \\
&= P(C)/ P(C)\cap P_{\mu^w} (C).
\end{align*}
Projection to the Levi quotient induces an affine fibration
$$
\text{pr}_w:
\Fc  (G,\mu) (C)^w  \longrightarrow \Fc  (M, \mu^w) (C).
$$

\begin{lemma}\label{lemma:reduction to P modification}
Suppose $\mu$ is minuscule. 
For $x\in \Fc  (G,\mu) (C)^w$ there is an isomorphism
$$
(\E_{b,x})_P \times_P M \simeq \E_{b_M,\mathrm{pr}_w (x)}.
$$
\end{lemma}
\begin{proof}
We use the Iwasawa decomposition 
$$
G(\BdR) = P (\BdR) G ( \BpdR).
$$
This induces an identification
$$
P(\BdR)/P(\BpdR) \iso G(\BdR)/G(\BpdR).
$$
Now if $y\mapsto x$ via this bijection then 
$$
(\E_{b,x})_P = \E_{b_P,y}.
$$
We use now that 
$$
G(\BpdR) = \coprod_{w\in W_P\backslash W/W_{P_\mu}} P(\BpdR) \overset{.}{w} P_\mu ( \BpdR)
$$
(this is deduced from the \'etaleness of the scheme $P_{F_\mu}\backslash G_{F_\mu}/ P_{\mu}$ where $F_\mu$ is the field of definition of $\mu$). Now consider 
$$
g \mu (t)^{-1} G (\BpdR) \in \Gr_{G,\mu}^{\BdR} (C)
$$
with $g\in G (\BpdR)$. Write $g= a \overset{.}{w} b$ with $a\in P(\BpdR)$ and $b\in P_\mu (\BpdR)$. We have 
$$
 g \mu(t)^{-1} = \underbrace{a  \mu^w (t)^{-1}}_{\in P(\BdR)}  \underbrace{\overset{.}{w}\mu(t) b \mu(t)^{-1}}_{\in G(\BpdR)}.
$$
The result is easily deduced.
\end{proof}

We can now state the main result of this section.

\begin{proposition}\label{P:wa}
Assume that $G$ is quasi-split and $[b]\in A(G,\mu)$ with $\mu$ minuscule. Then $x\in \Fc (G,\mu)(C)$  is weakly admissible if and only if for any standard parabolic $P$ with associated standard Levi $M$, any reduction $b_M$ of $b$ to $M$, and any $\chi\in X^\ast(P/Z_G)^+$, we have
 \[\deg\,        \chi_\ast (\E_{b, x})_P\leq 0,\] 
 where $(\E_{b,x})_P$  is the reduction to $P$ of $\E_{b, x}$ induced by the reduction $\E_{b_M}\times_{M}P$ of $\E_b$ and lemma \ref{P:redction modfication}.
\end{proposition}
\begin{proof}
We use lemma \ref{lemma:reduction to P modification}. The character $\chi$ factorizes through the Levi quotient $M$ and thus 
$$
\chi_* (\E_{b,x})_P = \chi_* \E_{b_M,\mathrm{pr}_w (x)}
$$ 
if $x$ is in the Schubert cell associated to $w$. Now, 
$$
c_1^G (\E_{b_M, \mathrm{pr}_w (x)}) = (\mu^{w})^\sharp - \kappa_M (b_M) \in \pi_1 (M)_\Gamma.
$$
In particular, for any $\chi\in X^*(M)$, 
$$
\deg \, \chi_*\E_{b_M, \mathrm{pr}_w (x)}  = \deg (\chi (b_M), \chi\circ \mu^w)
$$
where the degree of the right member has to be taken in the sense of filtered $\ph$-modules of rank $1$.
The result is then a consequence of the corrolary 9.2.30 of \cite{DOR} (together with lemma \ref{lemma:reduction to the case of maximal parabolic}).
\end{proof}


\section{The admissible locus}
\label{sec:ad locus}

As before, we fix the triple $(G, \{\mu\}, [b])$. We will assume that $\{\mu\}$ is \textsl{minuscule} until clearly stated otherwise and quickly assume $[b]\in B(G,\mu)$ (then the triple $(G, \{\mu\}, [b])$ is called a local Shimura datum, cf. \cite{RV} sec. 5).
\\

Let $K|\breve{E}$ be a finite extension and $x\in \Fc(G,\mu, b)^{wa}(K)$. We get the following  diagram of functors
\[\xymatrix@C=12mm{ \Rep\, G \ar[r]^-{\I_{b,x}}\ar[rd]_{\omega_G} & \varphi\mathrm{-}\Fil\Mod_{K/\Fb}^{wa}\ar[r]^-{V_{cris}}& \Rep_{cris}\Gal(\ov{K}/K)\ar[ld]^{\omega_{cris}}\\
	& \Vect_{F}&
	}\]
	where $\varphi\mathrm{-}\Fil\Mod_{K/\Fb}^{wa}$ is the category of weakly admissible filtered isocrystals, \\ $\Rep_{cris}(\Gal(\ov{K}/K))$ is the category of crystalline representations of $\Gal(\ov{K}/K)$ with coefficients in $F$ whose Sen operator is $F$-linear,   $\omega_G$ and $\omega_{cris}$ are the natural fiber functors. This diagram is commutative
since the class of the torsor $\underline{\mathrm{Isom}}^\otimes (\omega_G,\omega_{cris})$  is
given by $\kappa (b)-\mu^{\sharp}$ and is thus trivial (\cite{RV} prop. 2.7, see also 1.20 of \cite{RZ} in case $G_{der}$ is simply connected, and \cite{W} in the general case). Thus, the choice of an isomorphism between $\omega_G$ and $\omega_{cris}$ induces the homomorphism
\[\rho_x: \Gal(\ov{K}/K)\lra G(F).\]

Rapoport and Zink conjectured the existence of an open subspace \[\Fc (G,\mu, b)^{a}\subset  \Fc (G,\mu, b)^{wa}\] together with an \'etale $G$-local system $\mathscr{E}$ on it
such that for all finite extension $K|\breve{E}$ we have $\Fc(G,\mu, b)^{a}(K)=\Fc(G,\mu, b)^{wa}(K)$, and  $\mathscr{E}$ interpolates the preceding Galois representation $\rho_x: \Gal(\ov{K}/K)\ra G(F)$, cf. \cite{DOR} conj. 11.4.4 and \cite{Har} conj. 2.3. The associated spaces of lattices will give the desired tower of local Shimura varieties, (\cite{Ra} hope 4.2, \cite{RV} conj. 5.1, \cite{Sh} sec. 3.1).
\\

Now we will use the theory of $G$-bundles on the Fargues-Fontaine curve to define the admissible locus \[\Fc(G,\mu, b)^{a}\subset\Fc(G,\mu, b)^{wa}.\] 
 Let $C|\breve{E}$ be any complete algebraically closed extension containing $\Fbar$. We consider the curve $X=X_{C^\flat}$ over $F$ with the canonical point $\infty\in X$. 
\begin{definition}[See also \cite{Ra2} def. A.6]\label{D:admissible qs} 
\

\begin{enumerate}
\item
We set 
\[\Fc(G,\mu, b)^{a}(C)=\{x\in \Fc(G,\mu)(C)\;|\; \nu_{\E_{b,x}}\text{ is trivial} \}.\]
In other words, if $G$ becomes quasi-split over $F'|F$ we ask that $\E_{b,x}\otimes_F F'$ is semi-stable on the curve $X_{F'}$ and for all $\chi:G\rightarrow \G_m$, $\deg\, \chi_* \E_{b,x}=0$.
\item We define $\Fc (G,\mu,b)^{a}$ as the subset of $\Fc (G,\mu)$ stable under generalization whose $C$-points are given by the preceding for any $C$ as before.
\end{enumerate}
\end{definition}
Since $$c_1^G (\E_{b,x})= \mu^{\sharp}- \kappa (b)$$ the isomorphism class of $\E_{b,x}$ does not depend on $x\in \Fc (G,\mu,b)^a (C)$. In particular, if $[b]\in B(G,\mu)$ one has 
$$
\Fc (G,\mu,b)^a (C) = \{ x\in \Fc (G,\mu) (C)\ |\ \E_{b,x}\text{ is trivial} \}.
$$
%
Let us recall some basic properties of $\Fc(G,\mu, b)^{a}$.
\begin{proposition}\label{prop: lieu admissible}
The admissible locus satisfies the following properties:
\begin{enumerate}
	\item $\Fc(G,\mu, b)^{a}$ is a partially proper open subset of $\Fc(G,\mu)$.
	\item We have the inclusion \[\Fc(G,\mu, b)^{a}\subset\Fc(G,\mu, b)^{wa},\] such that for any finite extension $K|\breve{E}$
	  \[\Fc(G,\mu, b)^{a}(K)=\Fc(G,\mu, b)^{wa}(K).\] In particular, $\Fc(G,\mu, b)^{a}\neq \emptyset$ if and only if $[b]\in A(G, \mu)$, cf. prop. \ref{P:wa non empty}.
	\item When $(G, \{\mu\}, [b])$ is a Hodge type local Shimura datum (see \cite{Sh} 3.2), then the subspace  $\Fc(G,\mu, b)^{a}$  coincides with those introduced by Hartl \cite{Har} (via the Robba ring $\wt{B}^\dagger_{rig}(C)$) and Faltings \cite{Fal} (via the crystalline period ring $B_{cris}(C)$).
\end{enumerate}
\end{proposition}
\begin{proof}
(1) follows from the work of Kedlaya-Liu \cite{KL}. For (2) and (3) see \cite{Ra2} Remarks A.5.
\end{proof}

For the more advanced reader here is how to prove point (1) in the preceding. 
One can  consider Scholze's $\BdR$ affine grassmanian $\mathrm{Gr}^{\BdR}_G$ over $\Spa (\breve{E})^\diamond$. This is the \'etale sheaf associated to the presheaf $(R,R^+)\mapsto G ( \BdR (R))/ G(\BpdR (R))$ on affinoid  perfectoid $\breve{E}$-algebras. There is a  Bialynicki-Birula morphism that is an isomorphism thanks to our minuscule hypothesis (\cite{CS}). Its inverse gives an embedding 
$$
\Fc (G,\mu)^{\diamond} \hookrightarrow \mathrm{Gr}^{\BdR}_G.
$$
Now for any $\Fb$-perfectoid space $S$ together with an element $x\in \mathrm{Gr}^{\BdR}_{G} (S)$ one can define $$\E_{b,x},$$ a $G$-bundle on the relative adic Fargues-Fontaine curve $X_{S^\flat}$. This is defined using the "degree one Cartier divisor" 
$$
S\hookrightarrow X_{S^\flat}
$$
and a gluing "\`a la Beauville-Laszlo" by modifying $\E_b$ on $X_S$. In fact, when $S=\Spa (R,R^+)$, ``the formal completion along this Cartier divisor" is $\mathrm{Spf} ( \BpdR (R))$. According to Kedlaya and Liu, 
$$
S^{ss}:=\{s\in S\ |\ \E_{b,x |X_{k(s)^\flat}} \text{ is semi-stable}\}
$$
is open in $S$. Now if $x$ is given by a morphism $S\rightarrow \Fc (G,\mu)$ that is quasicompact quasiseparated surjective, then $|S|\rightarrow |\Fc (G,\mu) |$ is a quotient map such that $S^{ss}$ is the pullback of $\Fc(G,\mu,b)^a$ and one concludes.
\\

Let us now state the following result since this is not written anywhere explicitly.
It says that one can construct the local Shimura varieties as covers (in the sense of de Jong \cite{DeJong}) of $\Fc (G,\mu,b)^a$ as conjectured by Rapoport and Zink.
 We suppose $\mu$ minuscule as before.

\begin{theorem}[Scholze]
Suppose $[b]\in B(G,\mu)$. Then one can construct a pro-\'etale $\underline{G(F)}$-local system on $\Fc (G,\mu,b)^{a}$ such that for any compact open subgroup $K\subset G(F)$, the moduli of its $K$-trivializations $\M_K (G,\mu,b)\drt \Fc (G,\mu,b)^a$ is represented by a rigid analytic space. 
\end{theorem}
\begin{proof}
Let $S\rightarrow \Fc (G,\mu,b)^a$ be as before. Kedlaya and Liu prove in \cite{KL} that rank $n$ vector bundles on $X_{S^\flat}$ that are geometrically fiberwise on $S^\flat$ semi-stable of slope zero are the same as rank $n$ pro-\'etale $\underline{F}$-local systems on $S^\flat$.
Using this one can prove that the preceding $G$-bundle $\E_{b,x}$ on $X_{S^\flat}$ gives rise to a pro-\'etale $\underline{G(F)}$-torsor on $S^\flat$
$$
\underline{\mathrm{Isom}} (\E_1, \E_{b,x})
$$
where $\underline{G(F)}= \underline{\mathrm{Aut}} (\E_1)$ (we refer to \cite{FS}).
By varying the test morphisms $S\drt \Fc (G,\mu,b)^a$ this defines a pro-\'etale $\underline{G(F)}$-torsor $\M_\infty\rightarrow \Fc (G,\mu,b)^{a,\diamond}$. 
 Now one uses 
\cite{Sch1} (pro-\'etale descent of  separated \'etale morphisms) that says that $\M_\infty /K$ is representated by a rigid analytic space separated and \'etale over $\Fc (G,\mu,b)^a$.
\end{proof}

\section{Hodge-Newton decomposability}
\label{sec:HN decomposability}

In \cite{HN} (Theorem 1.1 \& Lemma 2.5) He and Nie give a description of Kottwitz set $B(G,\mu)$ as a subset of a positive Weyl chamber via the Newton map. This is applied in \cite{GoHeNi} to give a criterion for HN decomposability in terms of the Hodge polygon $\mu^\diamond$.  
He's description relies on his description of $B(G)$ via an affine root system derived from Bruhat-Tits theory applied to $G_{\Fb}$.  In this section we give a description of $B(G,\mu)$ that only relies 
on the usual relative and absolute spherical  root systems associated to $G$. He and Nie  description of $B(G)$ 
is well adapted to problems concerning integral models of Rapoport-Zink spaces and their mod $p$ special fiber. The problem we are interested in concerns only  the generic fiber and there is no reason to use Bruhat-Tits theory and integral structure on Dieudonn\'e modules for this. We will use an approach developed by Chai in \cite{Ch} and make the link with He's work.
 The two really original results of this section are proposition \ref{prop:minute pour B G 0 nub mu moins 1} and corollary \ref{coro:fully HN dec ssi idem} that will be a key point in the proof of our main theorem.

\subsection{A description of a generalized Kottwitz set  in the quasi-split case}
\label{sec:description a la He}

In this  subsection $G=H$ is a quasi-split group over $F$. 
Since $C_G (\AA)=\TT$ the restriction of any root of $\TT$ to $\AA$ is a root and 
this induces a bijection (\cite{Springer}  15.5.3)
$$
\Phi /\Gamma \xrightarrow{\ \sim\ } \Phi_0.
$$
One verifies moreover that 
$$
\Delta /\Gamma \iso \Delta_0.
$$
Let us fix $\mu \in X_*(\TT)^+$ {\it not necessarily minuscule}. As before, we note $\mu^\diamond\in X_*(\AA)_\Q^+$ the Galois average of $\mu$ and $[\nu_b]\in X_*( \AA)_\Q^+=\mathcal{N}(G)$ for $[b]\in B(G)$.
\\

 For $\alpha\in \Delta_0$ and $\beta\in \Phi $ such that $\beta_{|\AA}=\alpha$, one has $\beta\in \Delta$.
We note $$\omega_\beta \in \langle \Phi \rangle _\Q$$ the corresponding fundamental weight, that is to say for $\gamma \in \Delta$, 
$$
\langle  \gc, \omega_{\beta} \rangle  = \delta_{\gamma,\beta}.
$$
Now, for $\a\in \Delta_0$, we set 
$$
\tilde{\omega}_\alpha = \sum_{\substack{ \beta\in \Phi \\ \beta_{|A}=\alpha }} \omega_\beta \in X^* (\TT)_{\Q}^{\Gamma} = X^* (\AA)_\Q.
$$
This satisfies the following  property: for $\gamma \in \Delta $ we have 
$$
\langle  (\gc)^\diamond, \tilde{\omega}_\alpha \rangle  = \begin{cases}
 0, \quad\text{ if } \gamma_{|A}\neq \alpha \\
 1, \quad \text{ otherwise}
\end{cases}
$$
where $(\gamma^{\vee})^{\diamond} \in \langle\Phic_0\rangle_\Q$ is the Galois average of $\gamma^{\vee}$.
\\

We will need the following generalized Kottwitz set later.

\begin{definition}
For $\e\in \pi_1 (G)_\Gamma$ and $\delta\in X_* (\AA)_\Q^+$ we set 
$$
B(G,\e,\delta   ) = \{ [b]\in B(G)\ |\ \kappa (b)=\e \text{ and } [\nu_b]\leq \delta\}.
$$
\end{definition}

Of course if this set is non empty then $\e \equiv \delta$ in $\pi_1 (G)_\Gamma\otimes \Q$. One has 
$$
B(G,\mu) = B ( G, \mu^\sharp, \mu^\diamond).
$$

\begin{proposition}\label{prop:calcul general B ( G , epsilon, delta)}
If $\e=\mu^\sharp$ then, as a subset of $\mathcal{N}(G)$, $B(G,\e,\delta)$ is the set of $v\in X_*(\AA)_\Q^+$ satisfying
\begin{itemize}
\item $\delta- v\in \langle \Phi_0^\vee \rangle_\Q$,
\item $\forall \a\in \Delta_0 \text{ s.t. } \langle v,\alpha\rangle \neq 0, \langle \delta - v   ,\tilde{\omega}_\a\rangle  \geq 0  \text{ and } \langle \mu^\diamond -v, \tilde{\omega}_\a \rangle \in \Z.$
\end{itemize}
\end{proposition}
\begin{proof}
One has 
$$
B(G)= \bigcup_{M} \mathrm{Im} \big ( B(M)_{basic} \rightarrow B(G) \big )
$$
where $M$ goes through the standard Levi subgroups ($G$ included). Moreover 
$$
\kappa_M: 
B(M)_{basic} \xrightarrow{ \ \sim\ } \pi_1 (M)_\Gamma .
$$
Via this isomorphism 
the slope morphism associated to an element of $B(M)_{basic}$ is given by $$\pi_1 (M)_\Gamma \otimes \Q = X_*(Z_M)_\Q ^\Gamma.$$
For $v \in B(G,\e,\delta)\subset X_*(\AA)_\Q^+$, by definition  $\delta -v\in \Q_{\geq 0} (\Phi^{\vee}_0)^+$. Now for $\alpha \in \Delta_0$ and $z\in \Q_{\geq 0} (\Phi^{\vee})^{+}$ 
$$
\langle z , \tilde{\omega}_\alpha \rangle \, \geq 0
$$
since
 $$\Q_{\geq 0} (\Phi^\vee_0)^+ = \Q_{\geq 0} \{ (\gc)^\diamond \ |\ \gamma \in \Phi^+ \}.
$$
One deduces that $\langle \delta-v, \tilde{\omega}_\alpha \rangle \,\geq 0$. Now consider an element of $B(M)_{basic}$ given by the class of some $\mu'\in X_*(\TT)$ in $\big [ X_*(\TT)/\langle \Phic_M\rangle \big ]_{\Gamma}$. One has a decomposition 
$$
X_*(\TT)_\Q = \langle \Phic_M\rangle _\Q \oplus \langle \Phi_M\rangle _\Q^\perp 
$$
that gives rise to a projection morphism
$$
\mathrm{pr}_M: X_*(\TT)_\Q \longrightarrow \langle \Phi_M\rangle _\Q^\perp .
$$
Then the slope morphism in $X_*(\AA)_\Q$ associated to our element of $B(M)_{basic}$ is given by
$$
v= \mathrm{pr}_M (\mu')^\diamond
$$
the Galois average of $\mathrm{pr}_M(\mu')$. 
We can suppose that $M$ is the centralizer of $v$. Then $v$ defines a parabolic subgroup $P_v$ with Levi subgroup $M$. One can find $w\in W^\Gamma$ such that $P_v  ^w$ is a standard parabolic subgroup and $M^w$ a standard Levi. Up to replacing $M$ by $M^w$ and $\mu'$ by $\mu'^w$ we can thus suppose that $v\in X_*(\AA)_\Q^+$. Suppose now that the image of our element in $B(M)_{basic}$ in $B(G)$ lies in $B(G,\e, \delta)$. 
 Then $\mu$ and $\mu'$ have the same image in $\pi_1(G)_\Gamma$, 
$$
\mu - \mu' \in  \langle \Phic  \rangle  + I_\Gamma.X_*(\TT)
$$
where $I_\Gamma\subset \Z [\Gamma]$ is the augmentation ideal. Let $\alpha\in \Delta_0$ be such that $\langle v,\alpha\rangle \neq 0$ and thus $\alpha\in \Delta_0\setminus \Delta_{0,M}$. If 
$$
\mu - \mu' = \sum_{\gamma\in \Delta } \lambda_\gamma \gamma^\vee + z,  \ \lambda_\gamma\in \Z,\ \ z \in I_\Gamma.X_*(\TT)
$$
we have 
\begin{eqnarray*}
\langle  \mu^\diamond - v , \tilde{\omega}_\alpha \rangle  &=& \sum_{\substack{\gamma \in \Delta  \\ \gamma_{|A}=\alpha}} \lambda_\gamma \in \Z. 
\end{eqnarray*}
This proves that $B(G,\e, \delta)$ is contained in the announced subset of $X_*(\AA)_\Q^+$.
\\
Reciprocally, let $v\in X_*(A)_\Q^+$ satisfying the conditions of the statement. Let $M$ be the centralizer of $v$. Define 
$$
\mu'= \mu - \sum_{\alpha \in \Delta_0\setminus \Delta_{0,M}} \langle \mu^\diamond -v , \tilde{\omega}_{\alpha} \rangle \gc_\alpha \in X_*(\TT).
$$
where $\gamma _\alpha$ is any element in $\Phi $ such that $\gamma_{\alpha|A}=\alpha$.
One checks that the image of $\mu'$ in $\pi_1(M)_\Gamma$ defines an element of $B(M)_{basic}$ whose image in $B(G)$ lies in $B(G,\e,\delta)$ and whose associated slope is $v$.
\end{proof}


One deduces the following description for the usual Kottwitz set.

\begin{corollary}\label{prop:description de B G mu cas quasideploye}
As a subset of $X_* (\AA)_\Q^+$, $B(G,\mu)$ equals
$$
\big \{ v\in X_* (\AA)_\Q^+\ \big |\ \mu^\diamond - v \in\,  \langle \Phic_0\rangle _\Q\ \mathrm{and}\ \forall \alpha\in \Delta_0 \text{ s.t. } \langle v,\alpha\rangle \neq 0,  \  \langle  \mu^\diamond-v , \tilde{\omega}_\alpha\rangle \in \mathbb{N} \big \}.
$$
\end{corollary}

Later we will need the following.

\begin{corollary} \label{coro:description B G nubmumoins1}
Suppose $[b]\in B(G,\mu)$ is the basic element. Let 
\[B (G , 0, \nu_b\mu^{-1}) := B (G , 0, \nu_b (w_0\mu^{-1})^{\diamond}).\]
Then we have
\begin{align*}
B (G , 0, \nu_b\mu^{-1}) = \big \{ v\in X_* (\AA)_\Q^+\ \big |\  & v \in\,  \langle \Phic_0\rangle _\Q\ \mathrm{and}\ \forall \alpha\in \Delta_0 \text{ s.t. } \langle v,\alpha\rangle \neq 0,  \\  
& \langle \nu_b-w_0\mu^\diamond -v  , \tilde{\omega}_\alpha\rangle \geq 0 
 \text{ and } \langle v ,\tilde{\omega}_{\a}\rangle \in \Z \big \}.
\end{align*}
\end{corollary}

\begin{remark}
The preceding set $B(G,0,\nu_b \mu^{-1})$ is the one denoted $B(G,\nu_b\mu^{-1})$ in \cite{Ra2}. This notation may be confusing since if $\nu_b=1$, for example if $G$ is adjoint, then $B(G,0,\nu_b\mu^{-1})$ is not equal to $B(G,\mu^{-1})$ is general. This is why we introduced this more precise notation. 
\end{remark}


\subsection{The non-quasi-split case}
\label{sec:description B G mu non quasi split case}

Suppose now $G$ is not necessarily quasi-split. 
We have
$$
H^1 (F,H_{ad})=  \pi_1 (H_{ad})_\Gamma = \big [ \langle \Phi  \rangle^\vee / \langle\Phi^\vee \rangle \big ]_\Gamma
$$
We see $\langle \Phi \rangle^\vee$ as a lattice in  $\langle \Phic \rangle_\Q$.
The isomorphism class of the  inner form $G$ is then given by the class of some element 
$$
\xi  \in \langle \Phi \rangle^\vee \subset \langle \Phic \rangle_\Q.
$$
Projection to the adjoint group induces a bijection (see \cite{Kot2} 4.11)
$$
B(G,\e,\delta)  \xrightarrow{\ \sim\ } B (G_{ad},\e_{ad},\delta_{ad}).
$$
Moreover 
$$
H^1 (F,H_{ad}) = B(H_{ad})_{basic}
$$
and the isomorphism class of $G$ as an inner form of $H$ is given by some $[b_G]\in B(H_{ad})_{basic}$ for which $G_{ad}=J_{b_G}$. There is then a bijection 
$$
B(G_{ad}) \xrightarrow{\ \sim\ } B (H_{ad})
$$
that sends $[1]$ to $[b_G]$. 
We  can thus see $B(G,\e,\delta)$ as a subset of $B(H_{ad})$.
We have moreover an identification $\pi_1 (G_{ad})=\pi_1 (H_{ad})$.
Via this bijection the following diagram commutes
$$
\begin{tikzcd}
B(G_{ad})  \ar[d, "\kappa_{G_{ad}}"] \arrow[r, "\sim"] & B (H_{ad}) \ar[d, "\kappa_{H_{ad}}"] \\
\pi_1 (H_{ad})_\Gamma \ar[r, "\bullet + \xi"] & \pi_1 (H_{ad})_{\Gamma}
\end{tikzcd}
$$
We thus have 
$$
B ( G,\e,\delta ) = B ( H_{ad}, \e_{ad} + \xi,\delta_{ad} ).
$$
From proposition \ref{prop:calcul general B ( G , epsilon, delta)} we deduce the following.

\begin{proposition}
If $\e =\mu^\sharp$,
as a subset of $X_*(\AA)_\Q^+$, the set $B(G,\e,\delta)$ is given by the vectors $v$ such that 
\begin{enumerate}
\item $\delta - v \in \langle \Phic_0\rangle_\Q$,
\item for all $\a\in \Delta_0$ such that 
$\langle v,\a\rangle\neq 0$ one has $\langle \delta -v, \tilde{\omega}_\a \rangle \, \geq 0$,
\item for all $\a\in \Delta_0$ such that $\langle v,\a\rangle \neq 0$, $ \langle \mu^\diamond + \xi^\diamond-v,\tilde{\omega}_\a\rangle \, \in \Z$.
\end{enumerate}
\end{proposition}

This specializes to the two following statements.

\begin{corollary}
As a subset of $X_*(\AA)_\Q^+$, the set $B(G,\mu)$ is given by the vectors $v$ such that 
\begin{enumerate}
\item $\mu^\diamond - v \in \langle \Phic_0\rangle_\Q$,
\item for all $\a\in \Delta_0$ such that $\langle v,\a\rangle\neq 0$ one has $\langle \mu^{\diamond}-v, \tilde{\omega}_\a \rangle \, \geq 0$,
\item for all $\a\in \Delta_0$ such that $\langle v,\a\rangle \neq 0$, $ \langle \mu^\diamond + \xi^\diamond-v,\tilde{\omega}_\a\rangle \, \in \Z$.
\end{enumerate}
\end{corollary}

\begin{corollary}\label{coro:vrai coro utilise}
Suppose $[b]\in B(G,\mu)$ is the basic element. Then 
\begin{align*}
B (G , 0, \nu_b\mu^{-1}) = \big \{ v\in X_* (\AA)_\Q^+\ \big |\  v \in\,  \langle \Phic_0\rangle _\Q\ \mathrm{and}\ \forall \alpha\in \Delta_0 \text{ s.t. } \langle v,\alpha\rangle \neq 0, \\  \langle \nu_b-w_0\mu^\diamond -v  , \tilde{\omega}_\alpha\rangle \geq 0 
 \text{ and } \langle  v - \xi^\diamond ,\tilde{\omega}_{\a}\rangle \in \Z \big \}.
\end{align*}
\end{corollary}

\begin{remark}
When $G$ is unramified the root system we use is identical to the one used by He and the statement of the preceding proposition is identical to He's one. This is not the case anymore in general, even if $G$ is quasi-split. One can compare He's result with ours using theorem 6.1 of \cite{HainesDuality} to obtain that the $(\tilde{\omega}_\a)_{\a\in \Delta_0}$ are exactly the $\omega_{\mathcal{O}}$ of \cite{HN}  when $\Ol$ goes through the set of $\s_0$-orbits of simple roots in the Bruhat-Tits \'echelonnage root system attached to $G_{\Fb}$. An analysis of the construction of this root systems shows that we can take $\xi=\s(0)$ with the notations of \cite{HN}. 
\end{remark}


\subsection{HN decomposability}
\label{sec:sub sec HN decomposability}

Let us recall the following definition. Here $G$ is not necessarily quasi-split. 
\begin{definition}\label{def:HN decomposable}
The set $B(G,\mu)$ is fully HN decomposable if for any non-basic $[b]\in B(G,\mu)$ there exists a standard strict Levi subgroup $M$ of the quasi-split inner form $H$ such that:
\begin{enumerate}
\item the centralizer of $[\nu_b]$ is contained in $M$,
\item 
 $\mu^\diamond-[\nu_b] \in \langle \Phic_{0,M}\rangle _\Q$.
\end{enumerate}
\end{definition}

In the quasi-split case we have the following equivalent definition. As before we suppose $\mu \in X_*(\TT)^+$ which defines a cocharater with values in $M$ for any standard Levi subgroup $M$ (the conjugacy class $\{\mu\}$ does not define a unique conjugacy class in such an $M$, it is important to fix this).

\begin{lemma}\label{lemma:def eq fully HN dec cas quasi deploye}
For $G$ quasi-split the following are equivalent:
\begin{enumerate}
\item $B(G,\mu)$ is fully HN decomposable,
\item for any non basic $[b]\in B(G,\mu)$ there exists a strict standard Levi subgroup $M$  containing $M_b$ the centralizer of $[\nu_b]$ such that  $[b_M]\in B (M,\mu)$,
\item for any non basic $[b]\in B(G,\mu)$ there exists a strict standard Levi subgroup $M$  containing $M_b$ the centralizer of $[\nu_b]$ such that $\kappa_M (b_M)= \mu^{\sharp}\in \pi_1(M)_\Gamma  $.
\end{enumerate}
where $b_M$ is the reduction of $b$ to $M$ deduced from its canonical reduction to $M_b$ and the inclusion $M_b\subset M$.
\end{lemma}
\begin{proof}
This is easily deduced from the fact that $\pi_1 (M)_{\Gamma, tor}\hookrightarrow \pi_1 (G)_{\Gamma,tor}$ since this is identified with the injective map $H^1 (F,M)\rightarrow H^1 (F,G)$
(\cite{SerreCohoGal} ex. 1 p. 136).
\end{proof}


Before going further, 
let us remark that for $\a\in \Delta_0$, $\tilde{\omega}_\a\in \Q_{\geq 0} \Phi^+_0$ since for $\beta\in \Delta $, $\omega_\beta\in \Q_{\geq 0}\Phi ^+$. In particular, $$\langle\mu^{\diamond}, \tilde{\omega}_\a\rangle\geq 0.$$
 We remark too that if $\xi\in \langle \Phi \rangle^\vee$ is as in the preceding section, then one can form $\langle \xi^\diamond , \tilde{\omega}_\a\rangle \in \Q$. The reduction modulo $\Z$ of this quantity depends only on the class of $\xi$ in $H^1(F,H_{ad})=\pi_1 (H_{ad})_\Gamma$ and this defines a character
$$
\langle (-)^\diamond, \tilde{\omega}_\alpha\rangle : H^1 (F,H_{ad}) \longrightarrow \Q/\Z.
$$
We note $\{.\}: \Q/\Z \rightarrow [0,1[$ the fractional part lift.
The proof of the following proposition is then strictly identical to the one of the equivalence between (1) and (2) in theorem 2.3 of \cite{GoHeNi} (cf. the proof of \ref{prop:minute pour B G 0 nub mu moins 1} for this type of proof with our notations).

\begin{proposition}[Minute criterion]
The set $B(G,\mu)$ is fully HN decomposable if and only if for all $\a\in \Delta_0$, $\langle \mu^\diamond,\tilde{\omega}_\a\rangle + \{\langle \xi^\diamond, \tilde{\omega}_\a\rangle\} \, \leq 1$.
\end{proposition}

In particular if $G$ is quasi-split this is reduced to the condition 
$$
\langle \mu^{\diamond},\tilde{\omega}_\alpha\rangle \leq 1.
$$
Of course, in this case,  since $\langle \mu^{\diamond}, \tilde{\omega}_\a\rangle = \langle \mu, \tilde{\omega}_\alpha \rangle$, this can be rephrased in the following way:
$$
\forall\, \Ol,\text{ a }\Gamma\text{-orbit in }\Delta, \ \ \sum_{\beta\in \Ol}     \langle \mu, \omega_\beta \rangle  \leq 1.
$$
%
%
%

We will need the following later. 
Let $[b]\in B(G,\mu)$ be the basic element.
The full HN decomposability notion extends immediately to the set $B(G,0, \nu_b\mu^{-1})$. 

\begin{proposition}\label{prop:minute pour B G 0 nub mu moins 1}
The set $B(G,0,\nu_b \mu^{-1})$ is fully HN decomposable if and only if for all $\a\in \Delta_0$ one has $\langle \mu^\diamond ,\tilde{\omega}_\a\rangle + \{ \langle \xi^\diamond, \tilde{\omega}_\a \rangle \} \leq 1$.
\end{proposition}
\begin{proof}
We use  corollary \ref{coro:vrai coro utilise}. We can suppose $G$ is adjoint and thus $\nu_b=1$.
 \\
 
First, let us notice that the condition of the statement is equivalent to
$$
\forall \a\in \Delta_0, \ \ \langle -w_0\mu^\diamond , \tilde{\omega}_\a\rangle + \{ - \langle \xi^\diamond,\tilde{\omega}_\a \rangle \} \leq 1.
$$
In fact, if $*$ is the involution of Dynkin diagram induced by $-w_0$ then this last condition is equivalent to 
$\langle \mu^\diamond ,\tilde{\omega}_{\a^*}\rangle + \{ \langle w_0\xi^\diamond ,\tilde{\omega}_{\a^*}\rangle \}\leq 1$. But $\{ \langle w_0 \xi^\diamond , \tilde{\omega}_{\a^*}\rangle \} = \{ \langle \xi^\diamond , \tilde{\omega}_{\a^*}\rangle \}$ since $w_0\xi -\xi \in \langle\Phi^\vee\rangle$.
 
Suppose thus that $\forall \a\in \Delta_0, \langle -w_0 \mu^\diamond ,\tilde{\omega}_\a \rangle + \{- \langle \xi^\diamond,\tilde{\omega}_\a\rangle \}\leq 1$ and let $v\in B(G,0,\mu^{-1})$ be non basic. If $\langle v, \a\rangle \neq 0$ then 
\begin{align*}
&
\langle -w_0 \mu^\diamond , \tilde{\omega}_{\a}\rangle + \{ - \langle \xi^\diamond ,\tilde{\omega}_\a \rangle\} \\
=& 
\underbrace{\langle -w_0\mu^\diamond  -v ,\tilde{\omega}_\a\rangle }_{\geq 0}  + \underbrace{\underbrace{\langle \xi^\diamond ,\tilde{\omega}_\a\rangle  + \{ - \langle \xi^\diamond ,\tilde{\omega}_\a\rangle  \} }_{\in \Z}  + \underbrace{\langle v - \xi^\diamond ,\tilde{\omega}_\a\rangle}_{\in \Z}}_{\geq 0\text{ thus }\in \mathbb{N}} \leq 1.
\end{align*}
Thus, if $\langle -w_0 \mu^\diamond -v , \tilde{\omega}_\a\rangle \neq 0$ then $\langle v,\tilde{\omega}_\a \rangle + \{ - \langle \xi^\diamond,\tilde{\omega}_\a\rangle \} =0$ 
and thus $\langle v,\tilde{\omega}_\a\rangle=0$
which is impossible since $v$ is dominant and $\langle v,\alpha\rangle\neq 0$. We thus have
$$
-w_0 \mu^\diamond -v  \in \langle \Phi_{0,M}^\vee \rangle_\Q
$$
where $\Delta_{0,M}= \Delta_0\setminus \{ \a\}$ for any $\a$ such that $\langle v,\alpha \rangle\neq 0$.

Reciprocally, suppose $\langle -w_0 \mu^\diamond,\tilde{\omega}_{\a_0} \rangle + \{-\langle \xi^\diamond ,\tilde{\omega}_{\a_0} \rangle \}
 >1$ for some $\a_0\in \Delta_0$. Let $v$ be such that $\langle v, \a\rangle=0$ if $\a\neq \a_0$  and 
 $$
 \langle v, \tilde{\omega}_{\a_0}\rangle =1 -   \{ - \langle \xi^\diamond ,\tilde{\omega}_{\a_0}\rangle \}.
 $$
 Since $\tilde{\omega}_{\a_0}\in \Q_{\geq 0} \Delta_0$, $v$ is dominant. Moreover, 
 $$
 \langle -w_0\mu^\diamond - v , \tilde{\omega}_{\a_0}\rangle >0
 $$
 and 
 $$
 \langle v -\xi^\diamond , \tilde{\omega}_{\a_0}\rangle = 1 - \langle \xi^\diamond ,\tilde{\omega}_\a \rangle  - \{ - \langle \xi^\diamond ,\tilde{\omega}_{\a_0}\rangle \} \in \Z.
 $$
 Thus, $v\in B ( G,0,\mu^{-1})$ and is not HN decomposable since its centralizer is the maximal Levi subgroup $M$ with $\Delta_{0,M}= \Delta_0\setminus \{\a_0\}$ and $\langle -w_0\mu^\diamond - v , \tilde{\omega}_{\a_0}\rangle \neq 0$.
 \end{proof}

\begin{remark}
An analysis of the proof of the preceding proposition shows that if $B(G,0,\nu_b\mu^{-1})$ is fully HN decomposable, then for all non basic $[b']$ in this set $\nu_b-w_0 \mu^\diamond - [\nu_{b'}]\in \langle \Phi_{0,M}^\vee \rangle_\Q$, where $M$ is the centralizer of $[\nu_{b'}]$. The same holds for $B(G,\mu)$. 
\end{remark}

Here is the corollary we will use. 
This is a key point in the proof of our main theorem.
We don't know a direct proof of this in the sense that there is \`a priori no direct relation between $B(G,\mu)$ and $B(G,0,\nu_b\mu^{-1})$ (or $B(J_b, \mu^{-1})$ ).

\begin{corollary}\label{coro:fully HN dec ssi idem}
The following are equivalent:
\begin{enumerate}
	\item  the set $B(G,\mu)$ is fully HN decomposable,
	\item the set $B(G,0,\nu_b \mu^{-1})$ is fully HN decomposable,
	\item the set $B(J_b, \mu^{-1})$ is fully HN decomposable.
\end{enumerate}
\end{corollary}

\section{Harder-Narasimhan stratification of the flag variety}

\label{sec:HN stratification of the flag variety}

\subsection{The twin towers principle (\cite{FalTwin}, \cite{Fal}, \cite{FTwin}) }

Let $[b]\in B(G)$ be a basic element. What we call the "twin towers principle" is the identification
$$
\Bun_{J_b} = \Bun_G,
$$
that is to say there is an equivalence of groupoids between $G$-bundles and $J_b$-bundles on the curve.
Here and in the following  we give sometimes statements that are true at the level of perfectoid $v$-stacks of bundles like in \cite{F} and \cite{FS} or objects like the diamond $\Fc (G,\mu)^\diamond$. Nevertheless the reader not familiar with those notions should not be frightened; at the end, for the proof of our main theorem, we only need the evaluation on $C$-points of those objects and he can work in this context.
\\

In fact, $J_b\times X$ is the twisted pure inner form of $G\times X$ obtained by twisting by the $G$-torsor $\E_b$,
$$
J_b \times X = \underline{\mathrm{Aut}} (\E_b)
$$
as a group over the curve.
 If $\E$ is a $G$-bundle on $X$ one associates to it the $J_b$-bundle
$$
\underline{\text{Isom}} (\E_b,\E).
$$
This is what Serre calls "torsion au moyen d'un cocyle" in sec. 5.3 of \cite{SerreCohoGal}. 
At the level of points of the preceding perfectoid v-stacks this gives the well known bijection
$$
B(J_b)\iso B(G)
$$
that sends $[1]$ to $[b]$.

\begin{example}
For $\lambda\in \Q$ the functor $\E\mapsto \underline{\Hom} ( \Ol (\lambda), \E)$ induces an equivalence
between semi-stable vector bundles of slope $\lambda$ and  vector bundles  equipped with an action of $D_\lambda^{\mathrm{op}}=\mathrm{End} (\Ol(\lambda))^{\mathrm{op}}$, where $D_\lambda$ is the division algebra with invariant $\lambda$. 
\end{example}

The identification $\Bun_{J_b}=\Bun_G$ respects modifications of a given type $\mu$ that is to say it identifies the corresponding Hecke stacks of modifications. 
Suppose $[b]\in B(G,\mu)$ is the basic element. Let $$[b'']\in B(J_b,\mu^{-1})$$ be the basic element, $[b'']\mapsto [1]$ via $B(J_b)\xrightarrow{\sim} B(G)$. One thus has $$J_{b''}=G.$$
The preceding considerations give an isomorphism of moduli spaces over $\mathrm{Spa} (\breve{E})^\diamond$:
$$
\begin{tikzcd}[column sep=small]
\text{modifications of type }\mu \text{ between }\E^G_b\text{ and } \E^G_1 \ar[d,"\sim"] \\ \text{modifications of type }\mu \text{ between } \E^{J_b}_1 \text{ and }\E^{J_b}_{b''} \ar[d,equal] \\
 \text{modifications of type }\mu^{-1} \text{ between } \E^{J_b}_{b''} \text{ and }\E^{J_b}_{1}.
\end{tikzcd}
$$
At the end this induces a $J_b(F)\times G(F)$-isomorphism of local Shimura varieties with infinite level 
$$
\M (G,\mu,b)_\infty \iso \M (J_b,\mu^{-1},b'')_\infty
$$
as pro-\'etale sheaves on $\Spa (\breve{E})$ (that are representable by diamonds). This fits into a twin towers diagram 
using the de Rham and Hodge-Tate period morphisms that allow us to collapse each tower on its base 
$$
\begin{tikzcd}[row sep=large,column sep=large]
\M (G,\mu,b)_\infty \ar[d,"\pi_{dR}", two heads]\ar[r,"\sim"]  \ar[rd,"\pi_{HT}" description]
\ar[d, dash, dotted , bend right=40, start anchor={[xshift=-12mm]}, end anchor={[xshift=-12mm]}, start anchor={[yshift=3mm]}, end anchor={[yshift=-3mm]},"G(F)"']
& \M (J_b,\mu^{-1},b'')_\infty \ar[d,"\pi_{dR}", two heads] 
\ar[ld,"\pi_{HT}" description]    \ar[d, dash, dotted, bend left=40, start anchor={[xshift=12mm]}, end anchor={[xshift=12mm]}, start anchor={[yshift=3mm]}, end anchor={[yshift=-3mm]},"J_b(F)"]
  \\
\Fc(G,\mu,b)^a & \Fc (J_b,\mu^{-1},b'')^a  
\end{tikzcd}
$$
where:
\begin{itemize}
\item $\M (G,\mu,b)_\infty$ classifies modifications of type $\mu$ between $\E_b^G$ and $\E_1^G$.
\item For such a modification its image by $\pi_{dR}$ is $x$ if $\E_1^G=\E_{b,x}^G$. Its image by $\pi_{HT}$ is $y$ if $\E_b^G=\E_{1,y}^G$. 
\item $\M (J_b,\mu^{-1},b'')_\infty$   classifies modifications of type $\mu^{-1}$ between $\E^{J_b}_{b''}$ and $\E^{J_b}_1$. 
\item For such a modification its image by $\pi_{dR}$ is $x$ if $\E_1^{J_b} = \E^{J_b}_{b'',x}$. Its image by $\pi_{HT}$ is $y$ is $\E^{J_b}_{b''} =\E_{1,y}^{J_b}$. 
\end{itemize}

We will extend this type of diagram outside the admissible locus in section \ref{sec:the HN stratification}.

\subsection{Computation of the modifications of $\E_b$}

Let $[b]\in B(G)$ be any basic element. 

 \begin{proposition}[ \cite{Ra2} A.10]\label{prop:description modification basique}
 As a subset of $B(G)$ 
 there is an equality
\[\{ \E_{b,x}|\; x\in  \Fc(G,\mu)(C)\} / \sim \ = B(G, \kappa (b)- \mu^\sharp, \nu_b\mu^{-1}). \]
\end{proposition}
\begin{proof}
Let $f:B(G)\xrightarrow{\sim} B(J_b)$.
According to the twin towers principle 
$$\{\E_{b,x}\ |\ x\in \Fc (G,\mu ) (C)\}/\sim \, =f^{-1}\big ( \{\E_{1,x}\ | \ x \in \Fc ( J_b,\mu)(C)\}/\sim\big ).$$
 Now, 
$$
\{ \E_{1,x}\ |\ x\in \Fc (J_b,\mu) (C)\}/\sim\  = \{ [b']\in B(J_b)\ |\ \exists y\in \Fc (J_b,\mu^{-1}) (C),\  \E_{b',y}\simeq \E_1\}.
$$
The condition on the right hand side means $y\in \Fc (J_b,\mu^{-1},b')^a$ and $\kappa (b') =  (\mu^{-1})^{\sharp}$.
Now,  $$\Fc ( J_b,\mu^{-1},b')^{a}\neq \emptyset \Leftrightarrow [b']\in A (J_b,\mu^{-1}) $$
(see \ref{prop: lieu admissible}). Thus, 
$$
\{ \E_{1,x}\ |\ x\in \Fc (J_b,\mu) (C)\}/\sim\ = B(J_b,\mu^{-1}).
$$ 
Moreover, via the identifications $\mathcal{N}(G)=\mathcal{N}(J_b)$ and $\pi_1 (G)=\pi_1 (J_b)$, we have
\begin{align*}
\nu \circ f^{-1} &= \nu_b + \nu \\
\kappa\circ f^{-1} &= \kappa (b) + \kappa.
\end{align*}
The result follows immediately.
\end{proof}

\subsection{The Harder-Narasimhan stratification}
\label{sec:the HN stratification}

Suppose now that $[b]\in B(G,\mu)$ is the basic element.
According to proposition \ref{prop:description modification basique} there is a stratification
$$
\Fc (G,\mu ) = \coprod_{[b']\in B ( G,0,\nu_b\mu^{-1})} \Fc (G,\mu,b)^{[b']},
$$
where $ \Fc (G,\mu,b)^{[b']}$ is a locally closed generalizing subset of the adic space $\Fc(G,\mu)$ that defines a locally spatial sub diamond of $\Fc(G,\mu)^\diamond$. Here the fact that each stratum is locally closed can be deduced from Kedlaya-Liu's semi-continuity theorem of the Harder-Narasimhan polygon (\cite{KL}). The open stratum is 
$$
\Fc (G,\mu,b)^{[1]} = \Fc (G,\mu,b)^{a},
$$
the admissible locus. 
\\

One can describe each stratum in the following way. Fix $[b']\in B(G,0,\nu_b \mu^{-1})$ and let $[b'']\in B(J_b, \mu^{-1})$ be the corresponding element. We note $\widetilde{J}_{b'}=\underline{\mathrm{Aut}} (\E_{b'})$ the pro-\'etale sheaf of automorphisms of $\E_{b'}$ on $\mathrm{Perf}_{\overline{\F}_q}$ that is to say $S\mapsto \mathrm{Aut} (\E_{b' |X_S})$. One has 
$$
\widetilde{J}_{b'} =   \widetilde{J}_{b'}^0\rtimes \underline{J_{b'} (F)} 
$$
where $ \widetilde{J}_{b'}^0$ is a connected unipotent diamond that is a succesive extension of effective Banach-Colmez spaces (see \cite{F} and  \cite{FS}). This is identified with the same object for $b''$
$$
\widetilde{J}_{b''}= \widetilde{J}_{b'}.
$$
Now let 
$$
\mathcal{T}\longrightarrow \Fc (G,\mu,b)^{[b'],\diamond}
$$
be the pro-\'etale sheaf of isomorphisms between $\E_{b'}$ and $\E_{b,x}$, $x\in  \Fc (G,\mu,b)^{[b'],\diamond}$. Using again a result of Kedlaya-Liu (\cite{KL}) one can check this is a $\widetilde{J}_{b'}$-torsor. Now, using the twin towers principle, $\mathcal{T}$ is identified with 
$$
\M ( J_b, \mu^{-1},b'')_\infty
$$
and the morphism to $\Fc(G,\mu)$ with the Hodge-Tate period morphism $\pi_{HT}$.

\begin{proposition}
If $[b'']\in B(J_b,\mu^{-1})$ corresponds to $[b']\in B(G,0,\nu_b\mu^{-1})$ then 
$$
\Fc (G,\mu,b)^{[b']} = \mathrm{Im} \big ( \M ( J_{b},\mu^{-1},b'')_\infty \xrightarrow{\ \pi_{HT} \ } 
\Fc(J_b,\mu) =\Fc (G,\mu)\big ).
$$
The morphism $\pi_{HT}$ is a $\widetilde{J}_{b''}$-torsor and thus
$$
\Fc (G,\mu,b)^{[b'],\diamond} \simeq   \M ( J_{b},\mu^{-1},b'')_\infty\, /\,  \widetilde{J}_{b''}.
$$
One has 
$$
\dim \Fc ( G,\mu,b)^{[b']} = \dim \Fc (G,\mu) - \langle [\nu_{b'}], 2\rho \rangle.
$$
where the dimension of $\Fc (G,\mu,b)^{[b']}$  is the maximal lenght of a chain of specializations in the locally spectral space  (\cite{Sch1} sec. 21).
\end{proposition} 

We can thus again collapse the tower on two different bases  
$$
\begin{tikzcd}[row sep=large,column sep=small]
& \M (J_b,\mu^{-1},b'')_\infty \ar[ld,"\pi_{dR}"']  
\ar[dl, dash, dotted , bend right=30, start anchor={[xshift=0mm]}, end anchor={[xshift=0mm]}, start anchor={[yshift=-1mm]}, end anchor={[yshift=0mm]},"J_b(F)"']
\ar[rd,"\pi_{HT}"]
\ar[rd, dash, dotted , bend left=30, start anchor={[xshift=-1mm]}, end anchor={[xshift=0mm]}, start anchor={[yshift=0mm]}, end anchor={[yshift=0mm]},"\widetilde{J}_{b''}=\widetilde{J}_{b'}"] \\
\Fc (J_b,\mu^{-1},b'')^a & & \Fc (G,\mu,b)^{[b'],\diamond}
\end{tikzcd}
$$

%
%

\section{Proof of the main theorem}
\label{sec:proof of the main theorem}

As before, we consider a triple $(G, \{\mu\}, [b])$ with $\{\mu\}$ minuscule. In sections \ref{sec:weak ad locus} and \ref{sec:ad locus} we have introduced two open subspaces $\Fc (G,\mu, b)^{a}\subset \Fc (G,\mu, b)^{wa}$ of $\Fc (G,\mu)$. In general, the inclusion $\Fc(G,\mu, b)^{a}\subset\Fc(G,\mu, b)^{wa}$ is strict, see \cite{Har1} Example 3.6, \cite{Har} Example 6.7. In \cite{Har} section 9 and \cite{Ra2} A.20,  Hartl and Rapoport asked when do we have
\[ \Fc(G,\mu, b)^{a}=\Fc(G,\mu, b)^{wa} \ ? \ \]
 For $G=\GL_n$, Hartl gave a complete solution of this question in Theorem 9.3 of \cite{Har}. We give a complete solution to this problem for any $G$ when $[b]$ is basic.

 Here is the main theorem of this article.

\begin{theorem}\label{theo:main theorem}
If  $\mu$ minuscule and $[b]\in B(G,\mu)$ is basic then the following are equivalent
\begin{enumerate}
\item $B(G,\mu)$ is fully HN decomposable
\item $\Fc(G,\mu,b)^{a}= \Fc(G,\mu,b)^{wa}$.
\end{enumerate}
\end{theorem}
\begin{proof}
We  first treat the case when  $G$ is quasi-split.
\\
\noindent
{\bf (1)$\boldsymbol{\Rightarrow}$(2) [quasi-split case].}\   Let $x\in\Fc (G,\mu) (C)\setminus \Fc (G,\mu,b)^{a} (C)$. We want to prove that $x\notin \Fc (G,\mu,b)^{wa} (C)$. Let $[b']\in B(G,0,\nu_b\mu^{-1})$ be such that 
$$
\E_{b,x} \simeq \E_{b'}
$$
(proposition \ref{prop:description modification basique}). Since $x$ is not admissible, and thus $b'$ non basic, according to corollary  \ref{coro:fully HN dec ssi idem} and lemma \ref{lemma:def eq fully HN dec cas quasi deploye}, there exists a strict standard Levi subgroup $M$ of $G$ containing $M_{b'}$ such that 
\begin{equation}\label{eq:bla bla bla}
\kappa_M (b'_M)\otimes 1 = \nu_b + [w_0.(-\mu) ]^\sharp\otimes 1 \in \pi_1 (M)_\Gamma\otimes \Q
\end{equation}
where $b'_M$ is the image in $M$ of the canonical reduction of $b'$ to the centralizer $M_{b'}$ of $[\nu_{b'}]$. Here we assume $\mu \in X_* (\TT)^+$ as usual and $w_0.\mu^{-1}$ is seen as a cocharater of $M$.  Let $P$ be the standard parabolic subgroup associated to $M$. By lemma \ref{P:redction modfication} the reduction $\E_{b'_P}$ of $\E_{b'}$ induces a reduction
$$
\E_{b, P}
$$
of $\E_b$ to $P$. Let $[\tilde{b}]\in B(M)$ be such that 
$$
\E_{\tilde{b}}\simeq \E_{b, P} \times_P M. 
$$
Let us note $\tilde{b}^G$ for the image of $\tilde{b}$ in $G$. We are going to prove that $[\tilde{b}^G]=[b]$ in $B(G)$.
\\ 

According to lemma \ref{lemma:reduction to P modification} there exists $\mu_1 \in W.\mu$ and $y\in \Fc ( M,\mu_1^{-1})$ such that 
$$
\E_{\tilde{b}}\simeq \E_{b'_M,y}.
$$
We can suppose $\mu_1$ is in the {\it negative} Weyl chamber associated to $M$.
In particular we have 
$$
\kappa_M (\tilde{b} ) =\kappa_M (b'_M)  + \mu_1^\sharp \in \pi_1 (M)_\Gamma. 
$$
Using equation (\ref{eq:bla bla bla}) this implies
\begin{equation}\label{eq: ega ega}
\kappa_M (\tilde{b} )\otimes 1 =   \nu_b+ [ w_0. (-\mu) ]^\sharp \otimes 1 + \mu_1^\sharp \otimes 1\in \pi_1 (M)_\Gamma\otimes \Q.
\end{equation}
One can identify 
$$
\Hom (P, \G_m)\otimes \Q = (\pi_1 (M)_\Gamma\otimes \Q)^\vee.
$$
Via this equality (\ref{eq: ega ega}) gives that for any $\chi : P/Z_G \rightarrow \G_m$, 
$$
\deg \, \chi_* \E_{b,P} = \langle w_0. (-\mu),\chi\rangle + \langle \mu_1,\chi\rangle.
$$
Since $\E_{b}$ is semi-stable one obtains
$$
\forall \chi \in X^* (P/Z_G)^{+,\Gamma}, \  \ \ \langle w_0. (-\mu),\chi\rangle + \langle \mu_1,\chi\rangle\leq 0.
$$
Lemma \ref{lemma:permet de montrer que mu1 et mu idem} that follows then shows that
$$
\mu_1=-
w_0. (-\mu). 
$$
Inserting this in equation (\ref{eq: ega ega}) we obtain 
$$
\kappa_M (\tilde{b})\otimes 1=\nu_b \in \pi_1 (M)_\Gamma\otimes \Q.
$$
We can now conclude that $[\tilde{b}^G] = [b]$ using lemma \ref{lemma:lemme clef} that follows.
\\

Now, since $M_{b'}\subset M$, if we choose $\chi \in X^*(P/Z_G)^{+,\Gamma} \cap \mathbb{N}.\Delta_0$, then 
$$
\deg \, \chi_* \E_{b'_M} >0.
$$
Hence by proposition \ref{P:wa} $x$ is not weakly admissible.
\\

\noindent
{\bf (2)$\boldsymbol{\Rightarrow}$(1) [quasi-split case].} 
According to corollary \ref{coro:fully HN dec ssi idem} $B(G,0,\nu_b \mu^{-1})$ is not fully HN decomposable. We now use the construction at the end of the proof of proposition \ref{prop:minute pour B G 0 nub mu moins 1}. Let $\a\in \Delta_0$ be such that $\langle \mu^\diamond ,\tilde{\omega}_{\a^*} \rangle >1$. Let $M$ and $P$ be the associated standard maximal Levi and parabolic subgroups, $\Delta_{0,M}= \Delta_0\setminus \{\a\}$. Let $b'_M\in B(M)_{basic}$ be such that $\kappa_M (b'_M)= (\beta^\vee)^\sharp$ with $\beta\in \Delta$ and $\beta_{|\AA}=\a$.
Let $b'$ be the image of $b'_M$ in $G$. Then $[b']\in B(G,0,\nu_b\mu^{-1})$ is not HN decomposable and the centralizer of $[\nu_{b'}]$ is $M$. Let us note 
$$
Z=\{ x\in \Fc (G,\mu) (C) \ |\ \E_{b,x}\simeq \E_{b'}\}.
$$
For $x\in Z$ suppose $\E_{b,x}$ is not weakly admissible. Then there exists a standard maximal parabolic subgroup $Q$, a reduction $b_{M_Q}$ of $b$ to $M_Q$ and $\chi \in X^* (Q/Z_G)^+$ such that 
$$
\deg \, \chi_* (\E_{b,x})_Q >0.
$$
According to theorem \ref{theoSchi} the vector 
\begin{align*}
v: X^* (P/Z_G)&\longrightarrow \Z \\
\chi &\longmapsto \deg\, \chi_*  (\E_{b,x})_Q 
\end{align*}
seen as an element of $\mathcal{N}(G)$ satisfies $v\leq \nu_{\E_{b,x}}$. One deduces that $Q=P$ and  $(\E_{b,x})_Q$ is the Harder-Narasimhan canonical reduction of $\E_{b,x}$. 
Let $\mu_1\in W.\mu$ be such that 
$$
\E_{b'_M} \simeq (\E_{b,x} )_P\times_P M \simeq \E_{b_M,y}
$$
with $y\in \Fc ( M,\mu_1)$. One then has
$$
\kappa_M ( b'_M) = \kappa_M(b_M) - \mu_1^\sharp \in \pi_1 (M)_\Gamma.
$$
Pushing forward this equality in $\pi_1 (M)_\Gamma\otimes \Q= X_*(Z_M)_\Q^\Gamma$ one obtains 
$$
[\nu_{b'}] = \nu_b - \mu_1^\sharp \otimes 1.
$$
This gives 
$$
 \mu_1^\sharp \otimes 1 = \nu_b - (\beta^\vee)^\sharp\otimes 1
\in \pi_1 (M)_\Gamma\otimes \Q.
$$
We now use the diagram at the end of section \ref{sec:the HN stratification}. 
Let $[b'']\in B (J_b)$ corresponding to $[b'] \in B(G)$. 
Let us look at 
$$
\begin{tikzcd}[column sep=1mm]
 & \M ( J_b, \mu^{-1}, b'')_\infty (C) \ar[ld, two heads, "\pi_{dR}"'] \ar[rd, two heads , "\pi_{HT}"] \\
 \Fc (J_b, \mu^{-1},b'')^a (C) & & \Fc (G,\mu,b)^{[b']} (C)=Z
\end{tikzcd}
$$
The twin towers principle extends to $P$ and $M$-torsors. More precisely, the reductions $b_M$ and $b_P$ define a Levi and parabolic  subgroup of $J_b$ i.e. $M$ and $P$ transfer to the inner form $J_b$. We still denote them $M$ and $P$.
Moreover $b''$ admits a reduction $b''_M$ to $M$. 
 Then 
$$
\pi_{dR} \big  ( \pi_{HT}^{-1} (x) \big )
$$
lies in the locus of points $z\in  \Fc (J_b, \mu^{-1},b'')^a (C)$ where $(\E_{b'',z})_P\times_P M \simeq \E_{b''_M,s}\simeq \E_1$ with $s\in \Fc ( M,\mu_1^{-1})$. The open Schubert cell in $\Fc ( J_b,\mu^{-1})$ with respect to the action of $P$ is $P w_0 P_{\mu^{-1}}/P_{\mu^{-1}}$. The point $z$ is thus in this Schubert cell if and only if 
$$
-\mu_1\otimes 1 \equiv - w_0.\mu \otimes 1 \in \pi_1 (M) \otimes \Q.
$$
Projected to $\pi_1(M)_\Gamma\otimes \Q$ this is equivalent to 
$$
w_0.\mu \otimes 1 +  \beta^\vee \otimes 1  \equiv \nu_b.
$$
But this is impossible since 
$$
\langle -w_0.\mu -\beta^\vee, \tilde{\omega}_\a\rangle >0
$$
and is thus non zero.
\\

From this analysis we deduce that if $Z\subset \Fc ( G,\mu) \setminus \Fc (G,\mu,b)^{wa}$ then 
$$
\mathrm{Im} ( \pi_{dR}) =
\Fc ( J_b,\mu^{-1},b'')^a
$$
is contained in a profinite (index by $J_b (F)/P(F)$) union of non-open Schubert cells in $\Fc ( J_b,\mu^{-1})$. This is in contradiction with the openness of the admissible locus.
\\

Let us now explain how to treat the case of a general $G$ non necessarily quasi-split. One can suppose $G$ is adjoint. In fact, $\Fc ( G,\mu,b)^{wa}=\Fc ( G_{ad},\mu_{ad},b_{ad})^{wa}$ and 
$\Fc ( G,\mu,b)^{a}=\Fc ( G_{ad},\mu_{ad},b_{ad})^{a}$. Since $H$ is adjoint, $H^1(F,H)= B(H)_{basic}$ and $G$ is an extended pure inner form of $H$, $G=J_{b^*}$ with $[b^*]\in B(H)$ basic. Via $B(G)\xrightarrow{\sim} B(H)$ let $[b]\mapsto [b^H]$. Then 
\begin{align*}
\Fc ( G,\mu,b)^{wa} &= \Fc ( H,\mu, b^H)^{wa}, \\
\Fc ( G,\mu,b)^a & = \Fc (H,\mu, b^H)^a.
\end{align*}
Here $[b^H]\in B ( H, \mu^\sharp + \kappa (b^*) , \mu^{\diamond})$ is the basic element.  Via $B(G)\xrightarrow{\sim} B(H)$ the set $B(G,0,w_0\mu^{-1, \diamond})$ is sent to $B(H,\kappa (b^*),w_0\mu^{-1,\diamond})$. We deduce from corollary \ref{coro:fully HN dec ssi idem} that $B(G,\mu)$ is fully HN decomposable if and only if $B(G,0,w_0\mu^{-1,\diamond})\xrightarrow{\sim} B(H,\kappa (b^*),w_0\mu^{-1,\diamond})$ is. Then one checks that all the preceding arguments are valid in this context, working with $B ( H, \mu^\sharp + \kappa (b^*) , \mu^{\diamond})$ and  $B(H,\kappa (b^*),w_0\mu^{-1,\diamond})$ instead of $B(H,\mu)$ and $B(H,0,w_0\mu^{-1,\diamond})$.
\end{proof}

In the following lemmas $G$ is quasi-split.

\begin{lemma}\label{lemma:permet de montrer que mu1 et mu idem}
Consider $\mu \in X_*(\TT)^+$, $M$ a standard Levi subgroup of $G$ with associated standard parabolic subgroup $P$ and $\mu_1\in W.\mu$ that is in the positive Weyl chamber associated to $M$. If for all $\chi\in X^* (P/Z_G)^{+,\Gamma}$ one has $\langle \mu - \mu_1, \chi\rangle \leq 0$ then $\mu_1=\mu$.
\end{lemma}
\begin{proof} The condition $\langle \mu - \mu_1, \chi\rangle \leq 0$ for all $\chi\in X^* (P/Z_G)^{+,\Gamma}$ implies that $\mu-\mu_1$ is a linear combination of positive simple coroots in $M$ with positive coefficients. The result follows since $\mu_1$ and $\mu$ are both $M$-dominant.
\end{proof}

\begin{lemma}\label{lemma:lemme clef}
If $\E$ be a semi-stable $G$-bundle on $X$ equipped with a reduction $\E_P$ to the standard parabolic subgroup $P$ with standard Levi subgroup  $M_P$ such that $c_1^{M_P} ( \E_P\times_P M_P)\equiv \nu_{\E}$ in $\pi_1 (M_P)_\Gamma\otimes \Q$, then $\E_{P}\times_P M_P$ is in fact a reduction of $\E$ to $M_P$. 
\end{lemma}
\begin{proof}
Let us begin with some generalities about reductions to parabolic subgroups.  
Let us forget momentarily the hypothesis of the statement. 
Let $\E$ be a $G$-torsor on $X$.
Let $Q$ be a standard parabolic subgroup of $G$ such that $Q\subset P$ with associated standard Levi subgroup  $M_Q$. Then $Q\cap M_P$ is a standard parabolic subgroup of $M_P$ whose standard Levi subgroup is $M_Q$ and all of them are of this type. 
The morphism
$$
Q\backslash \E \longrightarrow P\backslash \E
$$
is a locally trivial fibration with fiber $Q\backslash P=M_P\cap Q\backslash M_P$. If $\E_P$ is a reduction of $\E$ to $P$ corresponding to the section $s$ of $P\backslash \E\rightarrow X$ then 
the pullback by $s$ of the preceding fibration is 
$$
M_P\cap Q \backslash (\E_P\times_P M_P) \longrightarrow X.
$$
As a consequence there is a bijection between 
\begin{itemize}
\item reductions $\E_Q$ of $\E$ to $Q$
\item reductions $\E_P$ of $\E$ to $P$ together with a reduction $(\E_P\times_P M_P)_{M_P\cap Q}$ of $\E_P\times_P M_P$  to $M_P\cap Q$.
\end{itemize}
Let us come back to our statement. We first prove that $\E_P\times_P M_P$ is semi-stable. Thus, let $Q\subset P$ be as before and $(\E_P\times_P M_P)_{M_P\cap Q}$ be a reduction corresponding to the reduction $\E_Q$. One has 
$$
X^* (Q) \iso X^* (M_P\cap Q) \iso X^* (M_Q).
$$
For $\chi \in X^*(Q)$ one has 
$$
\chi_* \E_Q =  \chi_{| M_P\cap Q * } (\E_P\times_P M_P)_{M_P\cap Q}.
$$
Now, suppose $\chi_{|M_P\cap Q} \in X^* ( M_P\cap Q/Z_{M_P})^+$. Then one can write 
$\chi=\chi_1+\chi_2$ with 
$$
\chi_1 \in X^* (Q/Z_G)^+\ \text{ and }\ \chi_2 \in X^*(P/Z_{G}).
$$ 
Then,
$$
\deg\, \chi_{| M_P\cap Q * } (\E_P\times_P M_P)_{M_P\cap Q} = \underbrace{\deg\, \chi_{1*} \E_Q}_{\leq 0\text{ by s.s. of }\E} + \underbrace{\deg\, \chi_{2*} \E_P}_{\langle \nu_\E ,\chi\rangle=0} \leq 0.
$$
Thus, $\E_P\times_P M_P$ is semi-stable with slope $\nu_\E$. Now, as
$$
c_1^G \big (  (\E_P\times_P M_P) \times_{M_P} G\big ) = c_1^G (\E),
$$
one concludes that $ (\E_P\times_P M_P) \times_{M_P} G\simeq \E$ using the injectivity of
$$
(c_1^G,\nu): H^1_{\textrm{\'et}} (X,G)\longrightarrow \pi_1(G)_\Gamma \times \mathcal{N}(G).
$$
\end{proof}

\begin{remark}
For $\GL_n$ the preceding lemma says that if $\E$ is  a semi-stable vector bundle equipped with a  finite filtration $(\Fil^i \E)_{i\in \Z}$ whose graded pieces  satisfy  $\forall i \ \mu ( \mathrm{Gr}^i \E)=\mu (\E)$, then $\E \simeq \bigoplus_{i\in \Z} \mathrm{Gr}^i \E$. In fact, the category of slope $\mu (\E)$ semi-stable vector bundles is abelian. From this one deduce by induction on $i$ that the $\Fil^i \E$ and the $\mathrm{Gr}^i \E$, $i\in \Z$, are semi-stable of slope $\mu (\E)$. On concludes using that if $\Fc_1$ and $\Fc_2$ are semi-stable with $\mu ( \Fc_1)=\mu (\Fc_2)$ then $\mathrm{Ext}^1 (\Fc_1,\Fc_2)=0$.
\end{remark}

\section{Asymptotic geometry of the admissible locus}
\label{sec:asymptotic}

In this section we sketch some ideas about fully HN decomposable period spaces. 
We suppose $\mu$ minuscule, $[b]\in B(G,\mu)$ is the basic element and $B(G,\mu)$ is fully HN decomposable.
Let us set 
$$
\partial \Fc (G,\mu,b)^a = \Fc (G,\mu,b) \setminus \Fc (G,\mu,b)^a.
$$
One has the stratification
$$
\partial \Fc (G,\mu,b)^a = \coprod_{[b']\in B( G,0,\nu_b \mu^{-1})\setminus \{[1]\}} \Fc ( G,\mu,b )^{[b']}.
$$
Suppose $G$ is quasi-split.
There is then a bijection 
\begin{align*}
&\{
 \text{parabolic subgroups of } G \text{ which admit a reduction of } b\}/\sim \\
 \iso & \{\text{parabolic subgroups of }J_b\}.
\end{align*}
A parabolic subgroup of $G$ transfer to the inner form $J_b$ if and only if $b$ has a reduction to this parabolic subgroup. Now for each $x\in \partial \Fc (G,\mu,b)^a (C)$ the canonical reduction $(\E_{b,x})_P$ 
defines a reduction $(\E_b)_P$. According to the first part of the proof of theorem \ref{theo:main theorem} this corresponds to a reduction of $b_P$, $(\E_b)_P=\E_{b_P}$. 

The same type of analysis can be lead if $G$ is non-quasi-split, writing $G_{ad} = J_{b^*}$ with $[b^*]\in B (H_{ad})_{basic}$ as at the end of the proof of theorem \ref{theo:main theorem}. 
At the end one obtains the following result.

\begin{proposition}
Let $\mathcal{P}$ be a set of representatives of the conjugacy classes of proper parabolic subgroups of $J_b$. 
There is a $J_b(F)$-invariant stratification by locally closed generalizing subsets 
$$
\partial \Fc (G,\mu,b)^a = \bigcup_{P\in \mathcal{P}} \partial_P \Fc (G,\mu,b)^a 
$$
together with a $J_b(F)$-equivariant 
continuous map $\partial_P \Fc (G,\mu,b)^a \rightarrow J_b (F)/P(F)$. 
\end{proposition}

Thus, {\it $\Fc (G,\mu,b)^a$ shares a lot of similarities with hermitian symmetric spaces}: its boundary is parabolically induced. Another way to see this is via the {\it generalized Boyer's trick} aka Hodge-Newton decomposition (\cite{Boyer} for the original trick, \cite{HT} for its Shimura varieties variant, \cite{Mantovan1}, \cite{Mantovan2},  \cite{Sh2} for generalizations in the PEL case, \cite{GoHeNi} for the special fiber  in general, \cite{Hansen} and \cite{GaisinImai} for ``modern" versions in the context of local Shtuka moduli spaces). 
In fact for $[b']\in B(G,0,\nu_b\mu^{-1})\setminus \{[1]\}$ we have the $\tilde{J}_{b''}$-torsor (see sec. \ref{sec:the HN stratification})
$$
\M (J_b,\mu^{-1},b'')_\infty \xrightarrow{\ \pi_{HT}\ } \Fc (G,\mu,b)^{[b']}.
$$
Now $[b'']\in B(J_b,\mu^{-1})$ is HN decomposable (cf. corollary \ref{coro:fully HN dec ssi idem}) and this local Shtuka moduli space is parabolically induced.
\\

Let us now state the following conjecture about the existence of analogs of {\it Siegel domains} in hermitian symmetric spaces.

\begin{conjecture}\label{conj:ad egal was ssi domaine fondamental compact}
For $[b]\in B(G,\mu)$ basic with $\mu$ minuscule the following are equivalent:
\begin{enumerate}
\item $\Fc ( G,\mu, b)^a=\Fc (G,\mu,b)^{wa}$.
\item The exists a quasi-compact open subset $U\subset \Fc ( G,\mu,b)^a$ such that 
$ J_b (F).U =
\Fc (G,\mu,b)^a
$.
\end{enumerate}
\end{conjecture}

Condition (2) is equivalent to the fact that there exists a quasi-compact open subset $V\subset \M (G,\mu,b)_\infty$ such that $J_b (F)\times G(F).V=\M(G,\mu,b)_\infty$. 
 Reduction theory and the construction of Siegel domains can be done via stability conditions in the Arakelov setting (\cite{Grayson}). In our setting this can be done in some particular cases, in the PEL case, using integral semi-stability conditions for finite flat group schemes (\cite{F5}, see in particular corollary 11 where the link is made with the Lafaille/Gross Hopkins fundamental domain (\cite{Laffaille}, \cite{GH}). For example, the preceding conjecture is solved in \cite{Sh3} for $U(1,n-1)$ using this technique and methods developed in \cite{F6}.
\\

Finally let us point out that the fact that the boundary of those spaces is parabolically induced has some cohomological consequences. In fact one can prove the following.

\begin{theorem}\label{theo: pas demontre}
If $\pi$ is a smooth representation with $\overline{\Q}_\ell$ coefficients of $G(F)$, $\mathcal{G}_\pi= \M(G,\mu,b)_\infty \times_{G(F)} \pi $ the corresponding pro-\'etale 
$J_b(F)$-equivariant 
 local system on $\Fc ( G,\mu,b)^{a}$ then
one has an isomorphism 
$$
R\Gamma_c \big ( \Fc (G,\mu,b)^a_{\C_p}/J_b (F), \mathcal{G}_\pi\big )_{cusp} \iso R\Gamma \big ( \Fc (G,\mu)_{\C_p} / J_b(F), \mathcal{G}_\pi )_{cusp}
$$
in the derived category of smooth representations of $J_b (F)$ with $\overline{\Q}_\ell$ coefficients.
\end{theorem}

Let us give more details on the meaning of the preceding conjecture in the context of \cite{FS}. The representation $\pi$ defines a local system $\Fc_{\pi}$  on the smooth perfectoid $v$-stack
$$
\big [\, \Spa (\overline{\F}_q)/ \underline{G(F)}\, \big ] =\Bun_G^{0,ss},
$$
the semi-stable locus of the component $\{ c_1^G=0\}$  
 in the perfectoid $v$-stack $\Bun_G$. Let 
 $$
 j: \Bun_{G}^{ss,0}\hookrightarrow \Bun_G,
 $$
 an open immersion. Then according to \cite{FS}, $j_! \Fc_\pi$ is a reflexive sheaf. Consider the Hecke correspondence
 $$
 \begin{tikzcd}
  & \mathrm{Hecke}_\mu \ar[rd,"\overset{\rightarrow}{h}"] \ar[ld,"\overset{\leftarrow}{h}"'] \\
  \Bun_G & & \Bun_G\times \mathrm{Div}^1
 \end{tikzcd}
 $$
 where $\mathrm{Div}^1= \Spa (\breve{E})^\diamond /\ph^\Z$. 
Then, again according to \cite{FS}, the Hecke transform 
$$
 R\overset{\rightarrow}{h}_* \overset{\leftarrow}{h}^* j_!\Fc_\pi
$$
is again reflexive. Let 
$$
x_b: [\, \Spa (\overline{\F}_q ) / \underline{J_b (F)}\, \big ] = \Bun_G^{c_1^G= -\kappa (b),ss} \hookrightarrow \Bun_G.
$$
Then 
$$
x_b^* \big ( R\overset{\rightarrow}{h}_* \overset{\leftarrow}{h}^* j_! \Fc_\pi \big )= R\Gamma_c \big ( \Fc (G,\mu,b)^a_{\C_p}/J_b (F), \mathcal{G}_\pi\big )
$$
as an admissible representation of $J_b (F)$. Moreover  
$\mathbb{D} (j_! \Fc_\pi) = Rj_{*} \Fc_{\pi}$ (Verdier dual) and 
$$
x_b^* \big ( R\overset{\rightarrow}{h}_* \overset{\leftarrow}{h}^* Rj_* \Fc_\pi \big )= R\Gamma \big ( \Fc (G,\mu,b)^a_{\C_p}/J_b (F), \mathcal{G}_\pi\big )
$$
as an admissible representation. Now, if $\ph_\pi: W_E\longrightarrow \,^L G$ is the L-parameter of $\pi$, by definition {\it $\ph_\pi$ is cuspidal} if it is discrete ($S_{\ph_{\pi}}/ Z (\widehat{G})^\Gamma)$ is finite) and the image of the intertia by $\ph_\pi$ in $\widehat{G}$ is finite (i.e. the monodromy operator is trivial). 
When $G$, resp. $J_b$, is quasi-split, conjecturally, $\ph_{\pi}$ is cuspidal if and only if {\it all} elements of the L-packet of $\pi$, resp. the packet or representations of $J_b (F)$ associated to $\pi$ via generalized Jacquet-Langlands, are supercuspidal. Under this hypothesis {\it $\Fc_{\pi}$ is clean}
$$
j_! \Fc_{\pi} = Rj_* \Fc_{\pi}.
$$
In this case theorem \ref{theo: pas demontre} is thus immediate. 
Nevertheless it may happen that one element of the L-packet is supercuspidal and not the other one in which case $\Fc_{\pi}$ may not be clean anymore. {\it 
Theorem \ref{theo: pas demontre} then says that  even if this cleanliness hypothesis is not satisfied, the cohomological consequence for the associated basic RZ spaces is satisfied if $B(G,\mu)$ is fully HN decomposable.}

%
%

\end{document}